\documentclass[journal,twoside,web]{ieeecolor}

\usepackage{generic}
\usepackage{cite}
\usepackage{amsmath,amssymb,amsfonts}
\usepackage{mathrsfs}
\usepackage{algorithmic}
\usepackage{graphicx}
\graphicspath{{Figures/}}
\usepackage{textcomp}
\usepackage{caption,subcaption}
\usepackage[hidelinks]{hyperref}

\usepackage{blindtext}
\usepackage{color}
\usepackage{mathtools}
\usepackage{centernot}

\usepackage{tikz}
\usetikzlibrary{shapes,arrows}
\usetikzlibrary{decorations.markings,decorations.pathmorphing,
	decorations.pathreplacing}
\usepackage[american voltages]{circuitikz}
\usetikzlibrary{patterns,arrows,shapes,calc,positioning,calc,intersections}
\usetikzlibrary{decorations.markings}
\usetikzlibrary{arrows.meta,shapes.arrows}

\newtheorem{thm}{Theorem}
\newtheorem{defn}{Definition}
\newtheorem{fact}{Fact}
\newtheorem{prop}{Proposition}
\newtheorem{lem}{Lemma}
\newtheorem{rem}{Remark}
\newtheorem{cor}{Corollary}

\newtheorem{obs}{Observation}

\newcommand{\R}{\hat{R}}
\newcommand{\Q}{\hat{Q}}
\renewcommand{\'}{^\top}
\newcommand{\inv}{^{-1}}
\newcommand{\psd}{^\dagger}
\newcommand{\modif}{\color{black}}

\def\thickhline{\noalign{\hrule height1pt}}

\def\BibTeX{{\rm B\kern-.05em{\sc i\kern-.025em b}\kern-.08em
    T\kern-.1667em\lower.7ex\hbox{E}\kern-.125emX}}
\begin{document}
	
\title{Interconnection of (\textit{Q,S,R})-Dissipative Systems in Discrete Time}
\author{Andrea Martinelli, Ahmed Aboudonia, and John Lygeros
\thanks{Research supported by the European Research Council under the Horizon 2020 Advanced Grant No. 787845 (OCAL).}
\thanks{Andrea Martinelli, Ahmed Aboudonia, and John Lygeros are with the Automatic Control Laboratory, Swiss Federal Institute of Technology (ETH) Zurich, Physikstrasse 3, 8092 Zurich, Switzerland (e-mails: andremar@ethz.ch, ahmedab@ethz.ch, lygeros@ethz.ch). }
}

\maketitle

\begin{abstract}
Discrete-time systems cannot be passive unless there is a direct feedthrough from the input to the output. For passivity-based control to be exploited nevertheless, some authors introduce virtual outputs, while others rely on continuous-time passivity and then apply discretization techniques that preserve passivity in discrete time. Here we argue that quadratic supply rates incorporate and extend the effect of virtual outputs, allowing one to exploit dissipativity properties directly in discrete time. We derive {\modif decentralized (\textit{Q,S,R})-}dissipativity conditions for a set of nonlinear systems interconnected with arbitrary topology, so that the overall network is guaranteed to be stable. For linear systems, we develop dissipative control conditions that are linear in the supply rate {\modif matrices}. To demonstrate the validity of our methods, we provide numerical examples in the context of islanded microgrids.
\end{abstract}


\section{Introduction}
\label{sec:introduction}

Dissipativity theory, as introduced by J.C. Willems in the 1970's \cite{WillemsDissipativity}, concentrate on how dynamical systems store and exchange energy over time. Thanks to its compositional framework and its relation with Lyapunov stability theory, dissipativity is a powerful tool to study the stability of interconnected \textit{continuous-time} (CT) systems. Indeed, the dissipativity of large-scale systems can often be deduced from the dissipativity of individual subsystems and their interconnection \cite{PotaLocalDissipativity, ChopraPassivity, ArcakDissipativity, PulkitAutomatica}.

On the other hand, in disciplines such as biology, demography, ecology, economics, engineering, finance, or physics, many systems naturally evolve over \textit{discrete time} (DT) steps \cite{GalorDiscrete}. DT systems also arise whenever performing a digital implementation of a controller and, moreover, they are the main focus of modern control techniques such as reinforcement learning and model predictive control. Despite the efforts throughout decades, however, dissipativity theory for DT systems is still not as mature as its CT counterpart \cite{LinByrnesPassivity,ByrnesLossless,LinPassivity,NavarroLopezKYPDissipativity,SimpsonpPorcoDissipativity,BrogliatoDissipativity}, possibly due to the fact that the theory was adapted to the DT framework, instead of being redeveloped anew \cite[Ch. 9.2.2]{NavarroLopezThesis}.

A perhaps surprising but important fact is that DT systems without direct feedthrough from the input to the output cannot be passive. As systems without feedthrough are prevalent in science and engineering \cite{AstromFeedbackSystems}, for passivity-based control to be exploited despite of this, one can either introduce a virtual output \cite{LinByrnesPassivity,AhmedPassivity}, or rely on CT passivity and then apply discretization techniques that preserve passivity in the DT domain \cite{LailaDissipationSampling,StramigioliDiscretization,CostaPassivityPreservation,OishiPassivityDegradation}. A virtual output is an artificial transformation that is used in place of the original output to render the system passive with respect to the transformed variables, but does not come with a clear interpretation in terms of energy balance. 

Unfortunately, it is known that properties like passivity or stability can be lost under discretization \cite{BrogliatoDissipativity}. Significant efforts have been devoted to develop discretization techniques that preserve passivity under different conditions, usually obtained by selecting a small enough sampling time or by considering virtual outputs \cite{LailaDissipationSampling, StramigioliDiscretization, CostaPassivityPreservation, OishiPassivityDegradation}. Another problem that arises from discretization is how to preserve the CT model structure, which is crucial in the context of decentralized control. Indeed, most of the available techniques compromise the sparsity pattern of the matrices involved  \cite{FarinaDiscretization,SouzaDiscretization}. In general, if the system is nonlinear or partially unknown, preserving passivity properties or the model structure becomes an even more challenging task \cite{KazantisDiscretizationNonlinear}. Similar issues affect the port-Hamiltonian framework \cite{KotyczkaDTportHamiltonian}, even though efforts have been devoted to develop models directly in DT \cite{TalasilaPortHamiltonian}.

Our aim here is to develop decentralized dissipativity-based analysis and control methods directly for the DT model. The main contributions can be summarized as follows; (i) we show that quadratic supply rates incorporate and extend the concept of virtual output; (ii) we develop dissipative control conditions that are linear in the supply rate {\modif matrices}; (iii) we derive decentralized {\modif $(Q,S,R)$-}dissipativity conditions for a set of interconnected nonlinear DT systems that guarantee asymptotic stability of the network; (iv) we provide numerical examples in the context of islanded DC microgrids.

In Section~\ref{Sec:dissipativity} we show that passivity {\modif with respect to} virtual outputs are just a special case of $(Q,S,R)$-dissipativity {\modif with respect to} the true outputs, and we provide a necessary condition for $(Q,S,R)$-dissipativity in terms of the matrix $R$. In Section \ref{sec:feedbackdissipativity} we introduce the dissipative control problem for linear systems and we derive novel conditions that are linear in the supply rate {\modif matrices}. The interconnection model is introduced in Section~\ref{Sec:interconnection}, together with decentralized stability conditions based on {\modif individual} $(Q,S,R)$-dissipativity of the nonlinear subsystems. Finally, in Section~\ref{sec:microgrids}, we validate our theoretical results through numerical examples on islanded microgrid models. In the Appendixes the reader will find some important facts about matrix and graph theory, as well as novel results on Laplacian flows and bounds on the Laplacian matrix pseudoinverse.

\subsubsection*{Notation} We denote with $\mathbb{R}_{\ge 0}$ and $\mathbb{R}_{>0}$ the set of nonnegative and positive real numbers, respectively. Given a set $\mathcal{X}$ such that $0\in\mathcal{X}\subseteq \mathbb{R}^n$, we say that a function $f:\mathcal{X} \to \mathbb{R}$ with $f(0)=0$ is positive semidefinite (PSD) when $f(x)\in\mathbb{R}_{\ge 0}$ for all $x\in \mathcal{X}$, and positive definite (PD) when $f(x)\in\mathbb{R}_{>0}$ for all $x\in \mathcal{X} \setminus \{0\}$. Similarly, a symmetric matrix $A=A^\top\in\mathbb{R}^{n\times n}$ is PSD (resp. PD) when its quadratic form $f(x)=x^\top Ax$ is a PSD (resp. PD) function;  we denote this as $A\succeq 0$ (resp. $\succ 0$). The space of differentiable and twice-differentiable functions is denoted by $C^1$ and $C^2$, respectively. The smallest and largest eigenvalues of a {\modif symmetric} matrix $A=A^\top$ are denoted with $\lambda_{\text{min}}(A)$ and $\lambda_{\text{max}}(A)$, respectively, and its inertia is $\text{In}(A)=(\rho_{-},\rho_0,\rho_+)$ where $\rho_{-}$, $\rho_0$, and $\rho_+$ denotes the number of negative, zero, and positive eigenvalues of $A$, respectively. Finally, we represent a vector of ones with $\mathbf{1}$.

%%%%%%%%%%%%%%%%%%%%%%%%%%%%%%%%%%%%%%%%%%%%%%%%%%%%%%%%%%%%

\section{Dissipativity Theory for DT Systems} \label{Sec:dissipativity}

\begin{figure}
	\centering
	\begin{tikzpicture}[>=stealth']
		
		\node[rectangle,draw,minimum height= 2cm,minimum width= 3cm,label=90:Dynamical System,text width=2.5cm,align=center] (DTsys) {$\Delta$ storage \\ $V(x^+)-V(x)$};
		
		\node[above right of= DTsys,node distance=0.4cm] (circle) {};
		
		\draw [->,red!85!black,thick,domain=30:150] plot ({0.9*cos(\x)}, {0.9*sin(\x)});
		
		\draw [->,red!85!black,thick,domain=210:330] plot ({0.9*cos(\x)}, {0.9*sin(\x)});
		
		\node[left of= DTsys,node distance=3.5cm] (supplystart) {};
		
		\node[right of= DTsys,node distance=3.5cm] (dissipationend) {};
		
		\draw[->,decorate,decoration={snake,post length=2mm,amplitude=1mm,segment length=5mm},thick,red!85!black] (supplystart) -- node[above,black] {supply} node[below,black] {$s(y,u)$} (DTsys);
		
		\draw[->,decorate,decoration={snake,post length=2mm,amplitude=1mm,segment length=5mm},thick,red!85!black] (DTsys) -- node[above,black] {dissipation} node[below,black] {$\phi(x,u)$} (dissipationend);
		
	\end{tikzpicture}
	\caption{Power balance for a DT system. When the internal accumulation is less than the energy supplied, the system is said to be dissipative.}
	\label{Fig:energybalance}
\end{figure}
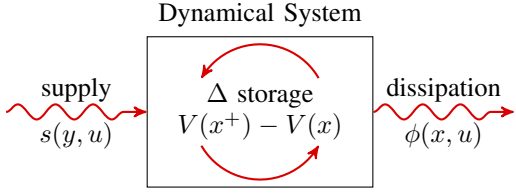

Consider the following nonlinear DT system,
\begin{subequations}\label{systemaffine}
	\begin{align}
		x^+ & = f(x,u) \label{dynamics}\\ 
		y & = h(x) \,, \label{output}
	\end{align}
\end{subequations}
where $x\in\mathbb{R}^n$ is the state, $u\in\mathbb{R}^m$ is the input or exogenous signal, $y\in\mathbb{R}^p$ is the output transformation, $f : \mathbb{R}^n \times \mathbb{R}^m \to \mathbb{R}^n$ is locally Lipschitz, and $h : \mathbb{R}^n \to \mathbb{R}^p$ is continuous with $f(0,0) = 0$ and $h(0)=0$.

 \begin{defn}\cite{ByrnesLossless} \label{def:dissipativity}
	The system \eqref{systemaffine} is \textit{locally dissipative} with respect to the supply rate $s:\mathbb{R}^p \times \mathbb{R}^m \to \mathbb{R}$, $s(0,0)=0$, when there exists a PSD storage function $V:\mathcal{X} \to \mathbb{R}_{\ge 0}$ such that, for all  $x,u \in \mathcal{X} \times \mathbb{R}^{m}$, 
	\begin{equation}\label{dissipation rate}
		V(x^+) - V(x) \le s(y,u)\,.
	\end{equation}
	If $\mathcal{X} = \mathbb{R}^{n}$ the system is simply called \textit{dissipative}. If $s(y,u) = y\' Q y + 2 y\'Su + u\' R u$ with $Q=Q^\top$ and $R=R^\top$, the system is said to be $(Q,S,R)$\textit{-dissipative}. When $m=p$ and $s(y,u) = y^\top u$, \textit{i.e.}, $(0,\tfrac{1}{2}I,0)$-dissipativity, the system is said to be \textit{passive}. The system is said to be \textit{strictly passive} if $V(x^+) - V(x) \le y^\top u - \omega(x)$ for some PD function $\omega : \mathbb{R}^n \to \mathbb{R}_{> 0}$ and \textit{output strictly passive} if $V(x^+) - V(x) \le y^\top u - \rho(y)$ for some $\rho : \mathbb{R}^m \to \mathbb{R}_{>0}$.
\end{defn}

{\modif It emerges from Definition \eqref{def:dissipativity} that $(Q,S,R)$-dissipativity is a special case of general dissipativity, and so is passivity with respect to $(Q,S,R)$-dissipativity.}
From the inequality \eqref{dissipation rate}, one can define the PSD dissipation rate  $\phi : \mathcal{X} \times \mathbb{R}^m \to \mathbb{R}_{\ge 0}$ as
\begin{equation}\label{powerbalance}
	\phi(x,u) = s(y,u) - V(x^+) + V(x) \,.
\end{equation}
Equation \eqref{powerbalance} can be interpreted as a power balance for the system \eqref{systemaffine}, as represented in Fig. \ref{Fig:energybalance}, where $s(y,u)$, $\phi(x,u)$, and $V(x^+) - V(x)$ are the power supplied, dissipated, and accumulated in the system, respectively.

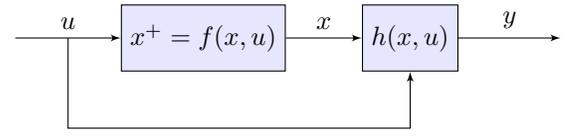
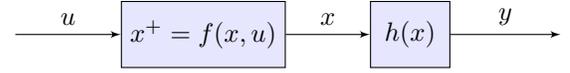
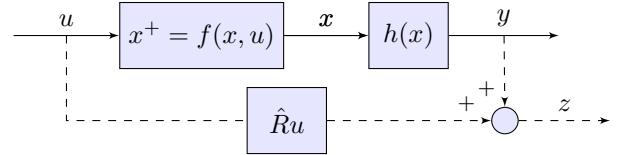
\begin{figure}
	\centering
	\begin{subfigure}[]{\columnwidth}
		\centering
		\begin{tikzpicture}[auto, node distance=2cm,>=latex']
			\tikzstyle{block} = [draw, fill=blue!10, rectangle, 
			minimum height=2.5em, minimum width=3em]
			\tikzstyle{input} = [coordinate]
			\tikzstyle{output} = [coordinate]
			\tikzstyle{middle} = [coordinate]
			\node [input, name=input] {};
			\node [block, right of = input, node distance=2.5cm] (system) {$x^+=f(x,u)$};
			\node [block, right of = system, node distance=2.75cm] (transf) {$h(x,u)$};
			\node [output, right of = transf] (output) {};
			\node [middle, below of = system, node distance=1.2cm] (middle) {};
			\draw [draw,->] (input) -- node [name=u] {$u$} (system);
			\draw [draw,->] (system) -- node {$x$} (transf);
			\draw [draw,->] (transf) -- node [name=y] {$y$} (output);
			\draw [draw,->] (u) |- node {} (middle) -| node {} (transf);
		\end{tikzpicture}
		\caption{Dynamical system with feedthrough.}
		\label{Fig feedthrough}
		\vspace{0.5cm}
	\end{subfigure}
	\vspace{0.5cm}
	\begin{subfigure}[]{\columnwidth}
		\centering
		\begin{tikzpicture}[auto, node distance=2cm,>=latex']
			\tikzstyle{block} = [draw, fill=blue!10, rectangle, 
			minimum height=2.5em, minimum width=3em]
			\tikzstyle{input} = [coordinate]
			\tikzstyle{output} = [coordinate]
			\tikzstyle{middle} = [coordinate]
			\node [input, name=input] {};
			\node [block, right of = input, node distance=2.5cm] (system) {$x^+=f(x,u)$};
			\node [block, right of = system, node distance=2.75cm] (transf) {$h(x)$};
			\node [output, right of = transf] (output) {};
			\node [middle, below of = system, node distance=1.2cm] (middle) {};
			\draw [draw,->] (input) -- node [name=u] {$u$} (system);
			\draw [draw,->] (system) -- node {$x$} (transf);
			\draw [draw,->] (transf) -- node [name=y] {$y$} (output);
		\end{tikzpicture}
		\caption{Dynamical system without feedthrough.}
		\label{Fig no feedthrough}
	\end{subfigure}
	\begin{subfigure}[]{\columnwidth}
		\centering
		\begin{tikzpicture}[auto, node distance=2cm,>=latex']
			\tikzstyle{block} = [draw, fill=blue!10, rectangle, 
			minimum height=2.5em, minimum width=3em]
			\tikzstyle{sum} = [draw, fill=blue!10, circle, minimum height=1em]
			\tikzstyle{input} = [coordinate]
			\tikzstyle{output} = [coordinate]
			\tikzstyle{middle} = [coordinate]
			\tikzstyle{virtual} = [coordinate]
			\node [input, name=input] {};
			\node [block, right of = input, node distance=2.5cm] (system) {$x^+=f(x,u)$};
			\node [block, right of = system, node distance=2.75cm] (transf) {$h(x)$};
			\node [output, right of = transf] (output) {};
			\node [middle, below of = system, node distance=1.2cm] (middle) {};
			\node [block, below right of = system, node distance=1.6cm] (Dmatrix) {$\R u$};
			\node [sum, right of = Dmatrix, node distance=2.9cm] (sum) {};
			\node [virtual, below of = sum, node distance=1.2cm] (virtual) {};
			\node [output, right of = sum, node distance=1.4cm] (outputz) {};
			\draw [draw,->] (input) -- node [name=u] {$u$} (system);
			\draw [draw,->] (system) -- node {$x$} (transf);
			\draw [draw,->] (transf) -- node [name=y] {$y$} (output);
			\draw [draw,->,dashed] (sum) -- node [name=z] {$z$} (outputz);
			\draw [draw,->,dashed] (u) |- node {} (Dmatrix) -- node [pos=0.85] {\footnotesize$+$} (sum);
			\draw [draw,->] (system) -- node {$x$} (transf);
			\draw [draw,->,dashed] (y) -- node [left,pos=0.75] {\footnotesize$+$} (sum);
		\end{tikzpicture}
		\caption{Dynamical system without feedthrough but with a virtual output.}
		\label{Fig feedthrough virtual}
	\end{subfigure}
	\caption{Block representation of a DT system (a) with feedthrough $y=h(x,u)$, (b) without feedthrough $y=h(x)$, and (c) with virtual output $z=y+\hat{R}u$. In the case (b), where $y$ does not depend directly on $u$, the system cannot be passive with respect to $y^\top u$. The system in configuration (c), instead, can be passive with respect to the virtual supply $z^\top u$.}
	\label{Fig:feedthrough}
\end{figure}

\subsection{Necessary and Sufficient Conditions}

 It is known that for DT systems to be passive, a necessary condition is that there exists a feedthrough term between the input and the output, as depicted in Fig. \ref{Fig feedthrough}, indicating that the control signal does influence the output directly \cite{LinByrnesPassivity, ByrnesLossless, LinPassivity}. On the other hand, systems without feedthrough as in Fig.~\ref{Fig no feedthrough} are the most commonly used in science and engineering for modelling dynamical systems \cite{AstromFeedbackSystems}. This motivates us to concentrate on systems without feedthrough like \eqref{systemaffine}. Note that the feedthrough is sometimes referred to as \textit{direct term} \cite{AstromFeedbackSystems}.

To understand why feedthrough is needed in the DT domain, consider the linear system $x^+=Ax+Gu$, $y=Cx+Du$, where $A\in\mathbb{R}^{n \times n}$, $G\in\mathbb{R}^{n \times m}$, $C\in\mathbb{R}^{p \times n}$, and $D\in\mathbb{R}^{p \times m}$. Then, for  $m=p$ and $s(y,u) = y\' u$, inequality \eqref{dissipation rate} is satisfied without loss of generality (see \cite{KottenstettePassivity}) if and only if there exists a PD quadratic storage function $V(x) = x\' Px$ such that
\begin{equation}\label{LMIpassivitylinear}
	\begin{bmatrix} A\' PA - P & A\' PG - \tfrac{1}{2}C\' \\ \star & D + G\' PG 
	\end{bmatrix} \preceq 0 \,.
\end{equation}
The linear matrix inequality (LMI) \eqref{LMIpassivitylinear} is feasible only if $D \preceq -G\' PG \preceq 0$, thus for systems without feedthrough ($D=0$) it does not admit any PD solution for $P$. 
On the other hand, when $y=Cx$ but $s(y,u)= y\' Q y + 2 y\'Su + u\' R u$, the LMI \eqref{LMIpassivitylinear} becomes
\begin{equation}\label{LMIdissipativitylinear}
	\begin{bmatrix} A\' PA - P - C\'QC & A\' PG - C\' S \\ \star & G\' PG - R 
	\end{bmatrix} \preceq 0 \,.
\end{equation}
In this case, $R$ takes the role of $D$, in the sense that the LMI \eqref{LMIdissipativitylinear} is verified only if $R\succeq G\' PG \succeq 0$. 

Necessary and sufficient conditions for a general nonlinear DT system to be dissipative have been derived in \cite{NavarroLopezKYPDissipativity} under the assumption that $V(f(x,u))$ and $s(y,u)$ are quadratic in $u$. On the other hand, necessary (but not sufficient) conditions for dissipativity under less restrictive assumptions are available in \cite{NavarroLopezKYPDissipativity} as a generalization of the necessary conditions for passivity developed in \cite{LinPassivity}. Consider the system \eqref{dynamics} with the output transformation $y=h(x,u)$ comprising a feedthrough, and assume the storage function and dissipation rate are in $C^2$. Then, since the dissipation rate is a PSD function, it is locally convex and $\tfrac{\partial^2 }{\partial u^2}\phi(x,u) \succeq 0$ in a neighbourhood of the origin. In case of passivity one has $s(y,u)=u^\top h(x,u)$ and, by differentiating \eqref{powerbalance} twice with respect to $u$, 
\begin{equation*}
	\begin{aligned}
		\frac{\partial^2 }{\partial u^2}&\phi(x,0) = \frac{\partial}{\partial u}h(x,0)+\frac{\partial}{\partial u}h^\top(x,0) \\
		& - \frac{\partial}{\partial u}f^\top(x,0) \left.\frac{\partial^2 V}{\partial \alpha^2} \right|_{\alpha=f(x,0)} \frac{\partial}{\partial u}f(x,0) \succeq 0 \,.
	\end{aligned}
\end{equation*}
Note that for systems without feedthrough $\frac{\partial}{\partial u}h(x,0)=0$, thus we are left with the necessary but contradictory condition $\frac{\partial}{\partial u}f^\top(x,0) \left.\frac{\partial^2 V}{\partial \alpha^2} \right|_{\alpha=f(x,0)} \frac{\partial}{\partial u}f(x,0) \preceq 0$. Indeed, since $\left.\frac{\partial^2 V}{\partial \alpha^2} \right|_{\alpha=f(x,0)} \succeq 0$, the previous inequality can only be satisfied with equality for some pathological cases. For example, in case of PD storage functions, a necessary condition for passivity would be that $\frac{\partial}{\partial u}f(x,0) = 0$ in a neighbourhood of the origin.
Next, {\modif we introduce a necessary condition for systems without feedthrough to be $(Q,S,R)$-dissipative.} 

\begin{lem}\label{corollary:R}
	{\modif The system \eqref{systemaffine} is $(Q,S,R)$-dissipative with $C^2$ storage function and $C^2$ dissipation rate only if $R\succeq 0$.}
\end{lem}

\begin{proof}
	First, we note that
	\begin{equation*}
		\frac{\partial^2}{\partial u^2} s(y,u) = \frac{\partial^2}{\partial u^2} ( y\' Q y + 2 y\'Su + u\' R u ) = 2R.
	\end{equation*}
	Since $\tfrac{\partial^2 }{\partial u^2}\phi(x,u) \succeq 0$ in a neighbourhood of the origin, differentiating \eqref{powerbalance} twice with respect to $u$ leads to
	\begin{equation*}
		2R \succeq \frac{\partial}{\partial u}f^\top(x,0) \left.\frac{\partial^2 V}{\partial \alpha^2} \right|_{\alpha=f(x,0)} \frac{\partial}{\partial u}f(x,0) \succeq 0 \,,
	\end{equation*}
	hence $R$ must be PSD. 
\end{proof}

This necessary condition in terms of the matrix $R$ generalizes the discussion on linear systems above to the nonlinear case. By inspecting the {\modif power balance \eqref{powerbalance} and the associated Fig. \ref{Fig:energybalance}, one can provide an energy interpretation:} for systems without feedthrough, $R \succeq 0$ is needed to supply nonnegative energy to the system to compensate the internal {\modif energy} accumulation due to the control input in the term $V(f(x,u)) - V(x)$.

\subsection{Virtual Output Interpretation}

To exploit passivity properties with systems that do not naturally have a feedthrough, one solution is to introduce a virtual output (see, \textit{e.g.}, \cite{LinByrnesPassivity} or \cite{AhmedPassivity}), that is, an artificial output transformation of the form 
\begin{equation} \label{virtualoutput}
	z=y+\R u \,,
\end{equation}
with $\R\in\mathbb{R}^{m \times m}$ (see Fig. \ref{Fig feedthrough virtual}). The circumflex symbol $\wedge$ is used to stress quantities that are related to virtual outputs. In this case, passivity is studied with respect to the virtual supply rate $z^\top u$ instead of the usual $y^\top u$.

We argue that another way to deal with systems without feedthrough is to consider {\modif richer supply rate functions than $y^\top u$}. Indeed, as mentioned for instance in \cite[Ch. 3.12.1]{BrogliatoDissipativity} and \cite[Rem. 7.5]{NavarroLopezThesis}, if a more general supply rate is considered, then the system may not have a feedthrough. In the following, we provide an energy interpretation of the virtual outputs, and show that {\modif quadratic} supply rate functions can incorporate and extend the effect of a virtual output.

\begin{obs}\label{Prop virtual output}

The system with virtual output \eqref{dynamics}-\eqref{virtualoutput} is passive if and only if the system with true output \eqref{dynamics}-\eqref{output} is $(0,\tfrac{1}{2}I,\R)$-dissipative. Moreover, the system with virtual output \eqref{dynamics}-\eqref{virtualoutput} is output strictly passive with a PD $\rho(y)=y^\top \Q y$ if and only if the system with true output \eqref{dynamics}-\eqref{output} is $(-\Q,\tfrac{1}{2}I,\R)$-dissipative.
\end{obs}

{\modif By decomposing} the virtual supply rate for passivity $z^\top u$ {\modif one can reveal} the relation with $(Q,S,R)$-dissipativity. Indeed,
		\begin{equation*}
			z \' u = (y + \R u)^\top u = \begin{bmatrix}
				y \\ u
			\end{bmatrix}\' \begin{bmatrix} 0 & \tfrac{1}{2}I \\ \star & \R
			\end{bmatrix}\begin{bmatrix}
				y \\ u
			\end{bmatrix} \,,
		\end{equation*}
		recovering the definition of supply rate corresponding to $(0,\tfrac{1}{2}I,\R)$-dissipativity for the system \eqref{dynamics}-\eqref{output}. Similarly, 
		\begin{equation*}
			z \' u - y^\top \Q y = \begin{bmatrix}
				y \\ u
			\end{bmatrix}\' \begin{bmatrix} -\Q & \tfrac{1}{2}I \\ \star & \R
			\end{bmatrix}\begin{bmatrix}
				y \\ u
			\end{bmatrix} \,,
		\end{equation*}
	which is the supply rate corresponding to $(-\Q,\tfrac{1}{2}I,\R)$-dissipativity for the system \eqref{dynamics}-\eqref{output}.

Observation \ref{Prop virtual output} demonstrates that for nonlinear systems \eqref{dynamics}, passivity/output strict passivity with respect to a virtual output is simply a special case of $(Q,S,R)$-dissipativity with respect to the actual output. We note that other connections between passivity with virtual outputs and $(Q,S,R)$-dissipativity may be established by using the same idea; for instance, one can consider \textit{input strict passivity}, \textit{output feedback passivity} or any other relevant definition available in the literature (see, \textit{e.g.}, \cite{KhalilNonlinear}). Here we focus on output strict passivity because it plays an important role when considering interconnected dissipative systems, as we will discuss in Section \ref{Sec:interconnection}. {\modif Finally, note that one may consider richer function classes for the supply rate to counteract the absence of feedthrough.  In the present paper, however, we focus on $(Q,S,R)$-dissipativity as a first step into this direction, that allows us to derive matrix inequality conditions that are linear in $Q$, $S$, and $R$, as discussed in the upcoming sections.}

\section{Dissipative Control for Linear DT Systems} \label{sec:feedbackdissipativity}

Consider the DT linear system described by $x^+ = Ax + Bv + Gu$ and $y = Cx$, where $B\in\mathbb{R}^{n\times r}$, $v\in\mathbb{R}^r$ represents a control input. The (quadratic) dissipative control problem consists in characterizing the state-feedback controllers of the form $v=Kx$, $K\in\mathbb{R}^{r \times n}$, such that the closed-loop system
\begin{subequations}\label{closed-loop}
	\begin{align}
		x^+ & = (A+BK)x+Gu \\
		y & = Cx \,,
	\end{align}
\end{subequations}
is $(Q,S,R)$-dissipative with respect to the pair $(u,y)$. The presence of the matrix $K$ as additional decision variable makes it less straightforward to reformulate the problem as an LMI (as in \eqref{LMIdissipativitylinear}). Moreover, we aim at obtaining an LMI formulation that is also linear in the matrices $Q$, $S$, and $R$. The reason is that, in Section \ref{Sec:interconnection}, we will introduce decentralized LMI conditions in $Q$, $S$, and $R$ that guarantee network stability, thus we want to treat them as additional decision variables.

\subsection{Primal Approach}

An LMI formulation for the dissipative control problem was proposed in 1999 by Tan and co-authors \cite[Thm. 2]{TanDissipativeControl}. The authors assume $Q$, $S$, and $R$ to be given, and the LMI is solved with respect to the storage function matrix $P$ and an auxiliary matrix $Z\in\mathbb{R}^{r \times n}$, such that the dissipative feedback gain can be retrieved as $K=ZP^{-1}$. Although linear in $P$ and $Z$, their reformulation is not linear in $Q$, $S$, and $R$, making it difficult to incorporate decentralized stability conditions. 
We propose here an alternative linear reformulation that allows one to also treat $Q$ and $R$ as decision variables.

\begin{lem} \label{lem:primalLMI}
	The closed-loop system \eqref{closed-loop} is $(Q,S,R)$-dissipative with $Q\prec 0$ if and only if there exist $P \succ 0$ and $Z$ such that 
	\begin{equation}\label{LMIprimal}
		\begin{bmatrix}
			P & AP + BZ & G & 0 \\
			\star & P & PC^\top S & PC^\top \\
			\star & \star & R & 0 \\
			\star & \star & \star & -\tilde{Q}
		\end{bmatrix} \succeq 0 
	\end{equation}
\end{lem}
\vspace{0.3cm}
holds with $\tilde{Q}\prec 0$, where $\tilde{Q} = Q^{-1}$. In that case, the dissipative feedback gain is $K=ZP^{-1}$ and an associated storage function is $V(x) = x^\top P^{-1}x$.  

\begin{proof}
 By considering a quadratic storage function $V(x) = x^\top \mathscr{P} x$ with $\mathscr{P}=\mathscr{P}^\top \succ 0$, one can infer from inequality \eqref{LMIdissipativitylinear} that the closed-loop system \eqref{closed-loop} is $(Q,S,R)$-dissipative if and only if
	\begin{equation*}
		\begin{bmatrix} A_K\' \mathscr{P} A_K - \mathscr{P} - C\'QC & A_K\' \mathscr{P} G - C\' S \\ \star & G\' \mathscr{P} G - R 
		\end{bmatrix} \preceq 0 \,,
	\end{equation*}
	with $A_K \coloneqq A+BK$. This inequality can be decomposed as 
	\begin{equation*}
		\begin{bmatrix} C\'QC + \mathscr{P} & C\' S \\ \star & R \end{bmatrix} 
		- \begin{bmatrix} A_K^\top \\ G^\top \end{bmatrix} \mathscr{P} \begin{bmatrix} A_K & G \end{bmatrix} \succeq 0 \,.
	\end{equation*}
	By taking the Schur complement \cite[Theorem 7.7.7]{HornJohnsonMatrix}, applying a congruence operation by pre- and post-multiplying by $\text{diag}(I, \mathscr{P}^{-1}, I)$, and letting $P\coloneq \mathscr{P}^{-1}$, one obtains
	\begin{equation*}
		\begin{bmatrix} P & A_K P & G \\
			\star & P & P C^\top S \\
			\star & \star & R \end{bmatrix} 
		- \begin{bmatrix} 0 \\ PC^\top \\ 0 \end{bmatrix} (-Q) \begin{bmatrix} 0 & CP & 0 \end{bmatrix} \succeq 0 \,.
	\end{equation*}
	Note that $A_K P = AP + BKP$. We define the auxiliary variable $Z \coloneqq KP$, so that the control gain can be uniquely determined by $K=ZP^{-1}$. Finally, by once again applying the Schur complement, one obtains \eqref{LMIprimal}.
\end{proof}

One may regard \eqref{LMIprimal} as a generalization of Lemma 1 in \cite{AhmedPassivity}, where they consider strict passivity of the closed-loop system with virtual output. We stress that \eqref{LMIprimal} is not linear in the $S$ matrix. By regarding $S$ as a decision variable the inequality \eqref{LMIprimal} becomes bilinear in $P$ and $S$. As bilinear problems are generally difficult to solve, we recommend to treat $S$ as a parameter and not a decision variable. In the next section, we discuss an alternative approach to the dissipative control problem to alleviate this shortcoming.

\subsection{Dual Approach} \label{sec:dualcontrol}

The content of this section is inspired by the dual approach to dissipativity discussed in \cite{HenkDissipativity} in the context of data-driven dissipativity analysis in open-loop. The terminology ``dual" refers to the fact that the Dualization Lemma (Fact \ref{Fact:Duality} in Appendix \ref{appendix:matrix}) is used to obtain a dissipativity condition for the dual of the system \eqref{closed-loop}, $x^+=-A_K^\top x - C^\top u$, $y=-G^\top x$.

\begin{lem} \label{lem:dissipativecontroldual}
 The closed-loop system \eqref{closed-loop} is $(Q,S,R)$-dissipative with $Q\prec 0$ and $R\succ 0$ if and only if there exist $P\succ 0$ and $Z$ such that
	\begin{equation} \label{LMIdual}
		\begin{bmatrix} P & (AP+BZ)^\top & PC^\top \\
			\star & P-G\mathscr{R}G^\top & G\mathscr{S} \\
			\star & \star & - \mathscr{Q} \end{bmatrix}  \succeq 0 
	\end{equation}
 holds with $\mathscr{Q}\prec 0$ and $\mathscr{R} \succ 0$, where $\begin{bsmallmatrix} \mathscr{Q} & \mathscr{S} \\ \star & \mathscr{R}
 \end{bsmallmatrix} = \begin{bsmallmatrix} Q & S \\ \star & R
\end{bsmallmatrix}^{-1}$. In that case, the dissipative feedback gain is $K=ZP^{-1}$ and an associated storage function is $V(x) = x^\top P^{-1}x$.  
\end{lem}

\begin{proof}
		By re-arranging the terms in \eqref{LMIdissipativitylinear}, one can show the closed-loop system \eqref{closed-loop} is $(Q,S,R)$-dissipative if and only if there exists $\mathscr{P} = \mathscr{P}^\top \succ 0$ such that
		\begin{equation} \label{dualizationLMI}
			\begin{bmatrix}
				I & 0 \\ 0 & I \\ A_K & G \\ C & 0
			\end{bmatrix}^\top \underbrace{\begin{bmatrix} -\mathscr{P} & 0 & 0 & 0 \\
					\star & -R & 0 & -S \\
					\star & \star & \mathscr{P} & 0 \\
					\star & \star & \star & -Q
			\end{bmatrix}}_{\coloneq\Psi} \begin{bmatrix}
				I & 0 \\ 0 & I \\ A_K & G \\ C & 0
			\end{bmatrix} \preceq 0 \,.
		\end{equation}
		If $Q\prec 0$ and $R\succ 0$, then $\text{In}(\begin{bsmallmatrix} Q & S \\ \star & R
			\end{bsmallmatrix}) = (p,0,m)$ by Haynsworth's Theorem \cite[4.5.P21]{HornJohnsonMatrix}, hence $\text{In}(\Psi) = (n+m,0,n+p)$.  
		\iffalse
		\begin{equation*}
			\Psi^{-1} = \begin{bmatrix} -\mathscr{P}^{-1} & 0 & 0 & 0 \\
				\star & -\mathscr{R} & 0 & -\mathscr{S} \\
				\star & \star & \mathscr{P}^{-1} & 0 \\
				\star & \star & \star & -\mathscr{Q}
			\end{bmatrix} \,.
		\end{equation*}
		\fi
		By the the Dualization Lemma (Fact \ref{Fact:Duality}), we conclude that \eqref{dualizationLMI} is verified with $Q\preceq 0$ if and only if 
		\begin{equation} \label{dualizationLMI2}
			\begin{bmatrix}
				-A_K^\top & -C^\top \\ -G^\top & 0 \\ I & 0 \\ 0 & I
			\end{bmatrix}^\top \Psi^{-1} \begin{bmatrix}
				-A_K^\top & -C^\top \\ -G^\top & 0 \\ I & 0 \\ 0 & I
			\end{bmatrix} \succeq 0 
		\end{equation}
		is verified with $\mathscr{R} \succeq 0$. We denote $P \coloneqq \mathscr{P}^{-1}$ and we decompose \eqref{dualizationLMI2} as
		\begin{equation*}
			\begin{bmatrix}
				P - G\mathscr{R}G^\top & G\mathscr{S} \\
				\star & - \mathscr{Q}
			\end{bmatrix} - \begin{bmatrix} A_K \\ C
			\end{bmatrix} P \begin{bmatrix} A_K^\top & C^\top \end{bmatrix} \succeq 0 \,.
		\end{equation*}
		Finally, the LMI \eqref{LMIdual} is obtained by substituting $Z=KP$ to the previous inequality, applying the Schur complement, and recognizing that $Q\prec 0, \; R\succ 0 \iff \mathscr{Q}\prec 0, \; \mathscr{R}\succ 0$ by matrix inversion properties \cite{BlockMatrixInversion}. Similarly, one can assume that \eqref{LMIdual} holds and, by noticing that $\text{In}(\Psi)=\text{In}(\Psi^{-1})$, it implies $(Q,S,R)$-dissipativity with $Q\prec0$ and $R\succ0$.
\end{proof}

\begin{rem}
	As mentioned in \cite{TanDissipativeControl}, a necessary condition for the dissipative control problem \eqref{LMIprimal} is that $(A,B)$ is a stabilizable pair. By duality theory \cite{AndersonStabilizability}, on the other hand, it follows that detectability of $(C,A)$ is necessary for \eqref{LMIdual}. 
\end{rem}

The main advantage of LMI \eqref{LMIdual} is that, besides linear in $P$ and $Z$, it is also linear in the dual supply rate matrices $\mathscr{Q}$, $\mathscr{S}$, and $\mathscr{R}$. Note that the primal variables can be uniquely determined by $\begin{bsmallmatrix} Q & S \\ \star & R
\end{bsmallmatrix} = \begin{bsmallmatrix} \mathscr{Q} & \mathscr{S} \\ \star & \mathscr{R}
\end{bsmallmatrix}^{-1}$. From a computational perspective, \eqref{LMIdual} is more attractive than \eqref{LMIprimal} since it involves an LMI of dimension $2n+m$ instead of $2n+m+p$.

We want to stress that $Q\prec0$ is a sufficient condition for the stability of the zero-input dynamics of $(Q,S,R)$-dissipative systems. Indeed, from \eqref{dissipation rate}, $V(f(x,0)) - V(x) \le y^\top Q y \le 0$, hence $x^+=f(x,0)$ is stable. In practice, since $Q\prec0$ will appear in Section \ref{Sec:interconnection} as a sufficient condition to guarantee network stability, no restriction is introduced by Lemmas \ref{lem:primalLMI} and \ref{lem:dissipativecontroldual} in this sense. Moreover, since $R\succeq0$ is a necessary condition for $(Q,S,R)$-dissipativity (see Lemma~\ref{corollary:R}), the additional requirement $R\succ 0$ in Lemma \ref{lem:dissipativecontroldual} reduces to $R$ being non-singular. In Section \ref{sec:discussion} we will derive decentralized LMI conditions in the dual variables that can be paired with \eqref{LMIdual} to guarantee network stability.

%%%%%%%%%%%%%%%%%%%%%%%%%%%%%%%%%%%%%%%%%%%%%%%%%%%%%%

\section{Interconnection of Dissipative DT Systems} \label{Sec:interconnection}

A simple digraph is a pair $\mathcal{G} = (\mathcal{V},\mathcal{E})$, where $\mathcal{V} = \{ 1,\ldots,N \}$ denotes the node set and $\mathcal{E} \subseteq \mathcal{V} \times \mathcal{V}$ the edge set. Let each node host a nonlinear DT subsystem with the same structure as \eqref{systemaffine},
\begin{subequations}\label{subsystemsaffine}
	\begin{align}
		x^+_i & = f_i(x_i,u_i) \\
		y_i & = h_i(x_i) \,,
	\end{align}
\end{subequations}
with $f_i(0,0)=0$ and $h_i(0)=0$ for all $i \in \mathcal{V}$. By denoting $u = \begin{bmatrix} u_1\' & \cdots & u_N\' \end{bmatrix}\'$ and $y = \begin{bmatrix} y_1\' & \cdots & y_N\' \end{bmatrix}\'$, we define the interconnection structure
\begin{equation} \label{linear interconnection}
	{\modif u=Hy\,.}
\end{equation}
The interconnection matrix $H\in\mathbb{R}^{mN \times pN}$ can represent a variety of linear interconnection structures. For instance, if $m=p$ and the subsystems \eqref{subsystemsaffine} are coupled together via 
\begin{equation}\label{coupling}
	u_i = \sum_{j\in\mathcal{N}_i} a_{ji}(y_j-y_i), \quad \forall i\in \mathcal{V}\,,
\end{equation}
then one can show that $H={\modif -} \mathcal{L} \otimes I \coloneq {\modif -} L$ (see, \textit{e.g.}, \cite{PulkitAutomatica}), where $\mathcal{L}=\mathcal{L}^\top$ is the weighted Laplacian matrix defined by the symmetric weights $a_{ji}=a_{ij}>0$, $\otimes$ denotes the Kronecker product, and $\mathcal{N}_i = \{ j \in \mathcal{V} \; : \; (i,j) \in \mathcal{E} \}$ is the set of neighbours of node $i$. The weighted degree associated to node $i$ is $d_i = \sum_{j=1}^{N} a_{ij}$, the lowest degree is $d_{\text{min}} = \min \{ d_1,\ldots,d_N \}$ and the degree matrix is $\mathcal{D}=\mbox{diag}(d_1,\ldots,d_N)$.

In the following, $H$ will be used to denote the general linear interconnection \eqref{linear interconnection}, while ${\modif -} L$ to denote the special case of Laplacian coupling \eqref{coupling}.
Finally, the network dynamics is 
\begin{equation}\label{global dynamics}
	x^+ = f(x,Hy) \coloneq \tilde{f}(x) \,,
\end{equation}
where $ x = \begin{bmatrix} x_1 \\ \vdots \\ x_N \end{bmatrix}$ and $f(x,u) = \begin{bmatrix} f_1(x_1,u_1) \\ \vdots \\ f_N(x_N,u_N) \end{bmatrix}$.

\subsection{Linear Systems and Virtual Outputs}

The authors in \cite{AhmedPassivity} infer asymptotic stability of the network \eqref{global dynamics} with Laplacian coupling starting from {\modif individual} strict passivity of the subsystems \eqref{subsystemsaffine} with respect to a virtual output. For clarity of presentation, we report the result here in our notation. Let $C = \mbox{diag}(C_1,\ldots,C_N)$, $\Q = \mbox{diag}(\Q_1,\ldots,\Q_N)$ and $\R = \mbox{diag}(\R_1,\ldots,\R_N)$, {\modif and recall that $\wedge$ refers to quantities associated with virtual outputs.}

\begin{lem} \cite[Lemma 2]{AhmedPassivity} \label{lemmaAhmed}
	Assume subsystems \eqref{subsystemsaffine} are coupled together via \eqref{coupling} and $f_i(x_i,u_i)=A_ix_i+B_iu_i$, $h_i(x_i)=C_ix_i$. Moreover, assume they are strictly passive with PD and $C^1$ storage functions $V_i:\mathbb{R}^n \to \mathbb{R}_{>0}$ and with respect to the virtual output $z_i = y_i + \R_i u_i$, that is,
	 \begin{equation*}
		V_i(x_i^+) - V_i(x_i) \le z_i^\top u_i - x_i^\top \Q_i x_i \quad \forall x_i,u_i, \; \forall i \in \mathcal{V} \,.
	\end{equation*}
	Then, if 
	\begin{equation}\label{globalconditionAhmed}
		C^{\top} L C - C^{\top}L^\top \hat{R} LC + \hat{Q} \succ 0 \,,
	\end{equation}
 the network \eqref{global dynamics} is asymptotically stable.
\end{lem}

We argue here that the global condition \eqref{globalconditionAhmed} can be simplified by considering a more appropriate definition of passivity -- output strict passivity instead of strict passivity -- together with the additional assumption that the subsystems are detectable. In this way, one can remove the dependence over the output transformation matrix $C$. 

\begin{lem} \label{Lem passivity linear}
	Assume subsystems \eqref{subsystemsaffine} are coupled via \eqref{coupling}, $f_i(x_i,u_i)=A_ix_i+B_iu_i$, $h_i(x_i)=C_ix_i$, and all the pairs $(A_i,C_i)$ are detectable. Moreover, assume they are output strictly passive with PD and $C^1$ storage functions $V_i:\mathbb{R}^n \to \mathbb{R}_{>0}$ and with respect to the virtual output $z_i = y_i + \R_i u_i$, that is, 
 \begin{equation*}
		V_i(x_i^+) - V_i(x_i) \le z_i^\top u_i - y_i^\top \Q_i y_i \quad \forall x_i,u_i, \; \forall i \in \mathcal{V} \,.
	\end{equation*}
	Then, the network \eqref{global dynamics} is asymptotically stable if
	\begin{equation}\label{inequalitylinear}
		L - L^\top \R L + \Q \succ 0 \,.
	\end{equation}
\end{lem}

\vspace{0.1cm}

\begin{proof}
		Let us define the candidate global Lyapunov function as the sum of the {\modif individual} PD storage functions, $V(x) \coloneq \sum_{i=1}^{N} V_i(x_i)$, which we assume to be quadratic. Then, 
	\begin{align*}
		V(x^+) - V(x) & \le  z^\top u - y^\top \Q y \\
		& = (y+\R u)^\top u - y^\top \Q y \\
		& = - y^\top (L-L^\top \R L + \Q) y \,,
	\end{align*}
where in the last equality we used $u=-Ly$. By imposing condition \eqref{inequalitylinear}, from Lyapunov stability theory for DT systems \cite{BofLyapunovDT} we infer that the trajectories of \eqref{global dynamics} will converge to the largest invariant set contained in $\Pi = \{ x : V(x^+) - V(x) = 0  \} = \{ x : y=0  \}$.
	Note that $y=0$ is equivalent to the network being decoupled, in the sense that no signal is exchanged among neighbouring subsystems (see \eqref{linear interconnection}); hence we can analyse their stability individually. Since the {\modif individual} subsystems are detectable by assumption, all solutions compatible with $y$ identically equal to zero will converge to the origin. To conclude, the largest invariant set contained in $\Pi$ is the origin, and \eqref{global dynamics} is asymptotically stable.
\end{proof}

The objective is now to find decentralized conditions on the matrices $\Q_i$ and $\R_i$ such that the global condition \eqref{inequalitylinear} is satisfied. Note that, although $\Q$ and $\R$ are block diagonal matrices, the sparsity of the Laplacian $L$ renders inequality \eqref{inequalitylinear} non-trivial to decouple. To tackle this issue, the authors in \cite[Theorem 1]{AhmedPassivity} propose a set of decentralized sufficient conditions based on diagonal dominance. In the following, we generalize those conditions to reflect the fact that inequality \eqref{inequalitylinear} does not depend on $C$ anymore. 

\begin{prop} \label{Fact local ineq}
	Inequality \eqref{inequalitylinear} is feasible if $\Q_i$ and $\R_i$ are diagonal matrices such that
\begin{equation*}
		0 \prec \R_i \prec \tfrac{1}{2d_i}I\,, \quad 0 \prec \Q_i \prec d_i I \quad \forall i \in \mathcal{V}\,.
\end{equation*}
\end{prop}

\vspace{0.2cm}

\begin{proof}
	Since $L\succeq 0$, \eqref{inequalitylinear} is implied by $\hat{Q}-L\hat{R}L\succ0$ which, by the Schur complement and if $\hat{R}\succ0$, is satisfied if and only if 
	
	\begin{equation*}
		\begin{bmatrix} \hat{Q} & L \\ \star & \hat{R}^{-1}
		\end{bmatrix} \succ0 \,.
	\end{equation*}
The result follows by considering diagonal dominance \cite[Theorem 6.1.10]{HornJohnsonMatrix}.
\end{proof}

The conservative assumption that $\Q_i$ and $\R_i$ have to be diagonal matrices comes from the use of diagonal dominance as decoupling method and, by the end of the section, we will discuss how this assumption can be lifted.  

\subsection{Nonlinear Systems}

A closer look to the proof of Lemma \ref{Lem passivity linear} suggests that linearity of the subsystems is not crucial to study how the energy is dissipated among interconnected DT systems. We first report a definition that extends detectability to nonlinear systems, {\modif initially introduced in \cite{ByrnesZSD} for CT systems and then extended in \cite{LinByrnesPassivity} to the DT case.} 

\begin{defn} 
	The nonlinear system \eqref{subsystemsaffine} is said to be zero-state detectable if all solutions of $x^+_{i}=f_i(x_i,0)$ that are identically contained in the set $\Pi_i=\{ x_i : y_i = 0 \}$ converge to the origin.
\end{defn}

Our aim is to use $(Q,S,R)$-dissipativity and the true output (instead of passivity and the virtual output) by exploiting the energy interpretation discussed in the previous section. Since virtual outputs are a special case of $(Q,S,R)$-dissipativity (Observation \ref{Prop virtual output}), we expect to obtain a more general stability condition for the network dynamics. Moreover, we now consider the more general interconnection \eqref{linear interconnection} instead of the Laplacian coupling \eqref{coupling}.
Let $Q = \mbox{diag}(Q_1,\ldots,Q_N)$, $S = \mbox{diag}(S_1,\ldots,S_N)$, and $R = \mbox{diag}(R_1,\ldots,R_N)$, {\modif and recall Definition \ref{def:dissipativity} for the definition of locally dissipative system.}

\begin{thm}\label{lem:stabilitynonlinar}
	Assume that the nonlinear systems \eqref{subsystemsaffine} are coupled via \eqref{linear interconnection} and locally $(Q_i,S_i,R_i)$-dissipative with PD and $C^1$ storage functions $V_i: \mathcal{X}_i \to \mathbb{R}_{>0}$. Then, if
	\begin{equation}\label{LMI global}
		{\modif \begin{bmatrix} I \\ H	\end{bmatrix}\' \begin{bmatrix} Q & S \\ \star & R 
			\end{bmatrix} \begin{bmatrix} I \\ H	\end{bmatrix} \prec 0 \,,}
 	\end{equation}
	the origin of the network \eqref{global dynamics} is stable. Moreover, if the systems \eqref{subsystemsaffine} are zero-state detectable, the origin is asymptotically stable. 
\end{thm}

\begin{proof}
	 First, note that $x=0$ is an equilibrium of the network \eqref{global dynamics}. We define the candidate global Lyapunov function $V(x) = \sum_{i=1}^{N} V_i(x_i)$ and the set $\mathcal{X} = \mathcal{X}_1 \times \dots \times \mathcal{X}_N$. Then, for all $x\in \mathcal{X}$, it holds 
	\begin{align*}
		V(x^+) - V(x) & \le  y\' Q y + 2y\'Su + u\' R u \\
		& = y\'(Q + SH + H^\top S^\top + H^\top RH )y \,,
	\end{align*}
    where we used $u=Hy$. By imposing condition \eqref{LMI global}, we note that $V(x)>0$ for all $x\in \mathcal{X} \setminus \{0\}$ and $V(x^+)-V(x) \le 0$ for all $x\in \mathcal{X}$. Thus $x=0$ is a stable equilibrium \cite{BofLyapunovDT}. Furthermore, since $V(x)$ is PD, there exists $\gamma>0$ such that the sublevel set $\Theta = \{ x\in\mathcal{X} : V(x) \le \gamma \}$ is compact and invariant for \eqref{global dynamics}. Hence, any trajectory of \eqref{global dynamics} starting from $x(0)\in\Theta$ will converge to the largest invariant set contained in $\Pi = \{ x \in \Theta : V(x^+) - V(x) = 0  \} = \{ x \in \Theta : y=0  \}$,
    which is again equivalent to the network being decoupled. If the subsystems are zero-state detectable, then all solutions identically contained in $\Pi$ will converge to the origin. Hence $x=0$ is an asymptotically stable equilibrium.
\end{proof}

\begin{cor}
	If the systems are (globally) {\modif $(Q_i,S_i,R_i)$-}dissipative, \textit{i.e.}, $\mathcal{X}_i =\mathbb{R}^{n}$ for all $i\in\mathcal{V}$, and the {\modif individual} storage functions are radially unbounded, \textit{i.e.}, $\| x_i \| \to \infty$ implies $V_i(x_i) \to \infty$ for all $i\in\mathcal{V}$, the origin is globally asymptotically stable.
\end{cor}

\begin{proof}
	The result follows from \cite[Thm. 1.4]{BofLyapunovDT}.
\end{proof}

As expected, $(Q,S,R)$-dissipativity allows one to derive a more general stability condition. Indeed, the stability condition \eqref{inequalitylinear} can be seen as a special case of \eqref{LMI global} when the coupling is Laplacian and the subsystems are $(-\Q_i, \tfrac{1}{2}I, \R_i)$-dissipative, exposing a parallel with Observation \ref{Prop virtual output}. {\modif However, it is important to remark that not all systems \eqref{subsystemsaffine} that are dissipative are also $(Q,S,R)$-dissipative, further emphasizing that the conditions in Theorem \ref{lem:stabilitynonlinar} are only sufficient.}

\begin{rem}
		Note that if each subsystem has access to the exact amount of dissipated energy via the {\modif individual} dissipation rates $\phi_i(x_i,u_i)$ as in \eqref{powerbalance}, then the global $\phi(x,u) = \sum_{i=1}^{N}\phi_i(x_i,u_i)$ can be used to further improve the feasibility of \eqref{LMI global} as $V(x^+)-V(x) = s(y,u) - \phi(x,u)$.
\end{rem}

{\modif \begin{rem}
	By taking a conical combination of the individual storage functions instead of simply their sum, the global Lyapunov function can be defined as $\tilde{V}(x)=\sum_{i=1}^{N}\sigma_i V_i(x_i)$, with $\sigma_i>0$ for $i=1,\ldots,N$, so that \eqref{LMI global} becomes $\tilde{Q}+\tilde{S}H+H^\top  \tilde{S}^\top + H^\top \tilde{R} H \prec 0$ with $\tilde{Q} = \mbox{diag}(\alpha_1Q_1,\ldots,\alpha_NQ_N)$, $\tilde{S} = \mbox{diag}(\alpha_1S_1,\ldots,\alpha_NS_N)$, and $\tilde{R} = \mbox{diag}(\alpha_1R_1,\ldots,\alpha_NR_N)$. A similar operation is proposed, \textit{e.g.}, in \cite[Ch. 2]{ArcakDissipativity} in CT settings, and it is useful to enlarge the feasibility of \eqref{LMI global} when $Q_i$, $S_i$, and $R_i$ are fixed. However, we do not consider it here since $Q_i$, $S_i$, and $R_i$ are decision variables themselves. 
\end{rem}}

 \begin{rem}
 	Theorem~\ref{lem:stabilitynonlinar} offers an estimate of the region of attraction of the origin (provided that the sublevel sets of $V(x)$ are compact in $\mathcal{X}$), that can be computed from the {\modif individual} $V_i$ and $\phi_i$ as $\Theta^* = \{ x\in\mathcal{X} : V(x) \le \gamma^* \}$, where
 	\begin{equation*}
 		\gamma^* = \max \, \gamma \;\; \text{s.t.} \;\; V(x) \le \gamma, \; x\in\mathcal{X}\,.
 	\end{equation*}
 \end{rem} 

\vspace{0.2cm}

Next, we aim to find decentralized conditions on the matrices $Q_i$, $S_i$, and $R_i$ such that \eqref{LMI global} is feasible. For this task we work under Laplacian coupling ($H={\modif -} L$), so that we can exploit the properties of the Laplacian matrix. Instead of approximating with diagonal dominance as in Proposition \ref{Fact local ineq}, however, we use the results introduced in Appendixes \ref{appendix:matrix} and \ref{appendix:graph}. Let us first define the parameters $\alpha \in \mathbb{R}$, $\tilde{\alpha} = \max \{ 1-\alpha, \; 0 \}$, and $\mathcal{S}\in\mathbb{R}^{m \times m}$. The proof of the next result can be found in Appendix \ref{appendix:theorem}.

\begin{thm}\label{TheoremBounds}
	The LMI \eqref{LMI global} is feasible under Laplacian coupling if at least one of the following decentralized LMIs is satisfied for all $i \in \mathcal{V}$,
\begin{subequations} \label{LMIlocalprimal}
	\begin{align}
		S_i=\tfrac{1}{2}\alpha I, \; 0 \prec R_i \prec \tfrac{1}{2d_i}I, \; Q_i \prec - 2  d_i \tilde{\alpha} I \,, \label{Thm1:first} \\ 
		S_i=\mathcal{S} \succeq 0, \; 0\prec R_i \prec \tfrac{1}{2d_i}I, \; Q_i \prec -2d_iI \,, \label{Thm1:second} \\ 
		0 \preceq S_i \prec \tfrac{1}{3d_i}I, \; 0 \prec R_i \prec \tfrac{1}{2d_i}I, \; Q_i + S_i \prec -4d_iI \,, \label{Thm1:third} \\ 
		S_i\succeq 0, \; R_i+S_i \prec \tfrac{1}{2d_i}I, \; Q_i\prec -2S_i, \; Q_i\prec -4d_iI \,. \label{Thm1:fourth}
	\end{align}
\end{subequations}
\end{thm}

\vspace{0.2cm}

Thanks to conditions \eqref{LMIlocalprimal}, each subsystem can {\modif individually} verify whether its dissipation properties guarantee the stability of their interconnection. Moreover, in case of linear systems, by combining \eqref{LMIlocalprimal} with \eqref{LMIprimal} each subsystem can design dissipative controllers that guarantee network stability. We stress that the only local parameter needed to evaluate the conditions \eqref{LMIlocalprimal} is the weighted degree $d_i$ of each node. The global parameters $\alpha$ and $\mathcal{S}$ in \eqref{Thm1:first} and \eqref{Thm1:second} need to be available in advance and be the same for each subsystem. Algorithms based on message exchange protocols, as reviewed in \cite{DCOPsurvey}, can be used to select those parameters in a distributed way. In Section \ref{sec:microgrids}, we will provide a numerical comparison of the four conditions.

\subsection{Discussion} \label{sec:discussion}

In the remainder of this section we discuss various aspects highlighting important points of our results.

\vspace{0.1cm}

\subsubsection*{Trivial Decomposition}\label{remarkR} 
	One could be tempted to simply {\modif infer that \eqref{LMI global} is feasible if (check Fact \ref{Fact:congruence} in Appendix \ref{appendix:matrix})}
	\begin{equation}\label{LMIwrong}
	{\modif	\begin{bmatrix} Q & S \\ \star & R 
		\end{bmatrix} \prec 0 \,. }
	\end{equation}
	However, \eqref{LMIwrong} holds only if $R\prec 0$ (\textit{i.e.}, $R_i \prec 0$ for all $i \in \mathcal{V}$), which violates the assumption that the subsystems are $(Q_i,S_i,R_i)$-dissipative in the sense of Lemma~\ref{corollary:R}. 

\vspace{0.1cm}

\subsubsection*{Comparison with Passivity and Virtual Output} \label{remark comparison}
	By setting $\alpha = 1$ in \eqref{Thm1:first} we obtain $S_i=\tfrac{1}{2}I$, hence we can directly compare our bound with the one in Proposition~\ref{Fact local ineq}, which was derived starting from output strict passivity with respect to a virtual output. In this case, with $\alpha = 1$, \eqref{Thm1:first} is equivalent to 
	\begin{equation*}\label{localcondition}
		0 \prec R_i \prec \tfrac{1}{2d_i} I\,, \quad Q_i \prec 0 \quad \forall i \in \mathcal{V}\,.
	\end{equation*}
By direct comparison, we can assess that this is a generalization of Proposition \ref{Fact local ineq}. Indeed, we lifted the conservative assumption that $Q_i$ and $R_i$ must be diagonal matrices, and got rid of the lower bound on $Q_i$ (equivalently, the upper bound on $\Q_i$). We speculate that this major improvement is possible both because we started from the more general stability condition \eqref{LMI global}, and because we employed decoupling techniques that do not make use of diagonal dominance, which is typically rather conservative. 

\vspace{0.1cm}

\subsubsection*{Role of Network Topology}
	By inspecting the stability condition \eqref{LMI global}, we realize that it only depends on the interplay between the ($Q,S,R$) matrices and the interconnection matrix $H$. One can then fix the matrices ($Q,S,R$) and infer which network property facilitates the stabilization problem. In case of Laplacian coupling, we note that both the opposite of the degree $-d_i$ and its inverse $\tfrac{1}{d_i}$ appear as upper bounds on $Q$, $S$, and $R$ in Theorem~\ref{TheoremBounds}. This suggests that, as one might expect, networks that are dense and with high interconnection weights tend to be more difficult to stabilize via decentralized control. 
	
On the other hand, by regarding \eqref{LMI global} as a quadratic matrix inequality (QMI) in the matrix variable $H$, one can conclude that \eqref{LMI global} is nonempty only if $S^\top R^\dagger S - Q \succ0$ \cite{HenkQMI}. Moreover, one could try and infer when $H$ facilitates the feasibility of the QMI, or even develop robust feasibility conditions for uncertain topologies in the spirit of \cite{HenkQMI}.
	
	On the other hand, when $H=0$ then \eqref{LMI global} reduces to $Q \prec 0$. This agrees with the fact that an isolated dissipative system is stable if $Q_i\prec 0$.  Indeed, as discussed at the end of Section \ref{sec:feedbackdissipativity}, the zero-input dynamics is stable if $Q_i\prec0$ and asymptotically stable under zero-state detectability assumption.
	
	Regarding the interconnection model \eqref{linear interconnection}, we already mentioned that the matrix $H$ can represent a variety of coupling structures. For instance, when $H=(\mathcal{A}-\mathcal{A}^\top) \otimes I$, the interconnection is called \textit{skew-symmetric} and can be used to model certain power distribution systems \cite{PulkitAutomatica}. Another fundamental example is the feedback interconnection of two dissipative systems, which is captured by $H=\begin{bsmallmatrix} 0 & I \\ I & 0 \end{bsmallmatrix}$. In this case, \eqref{LMI global} reduces to $\begin{bsmallmatrix} Q_1 + R_2 & S_1+S_2^\top \\ \star & R_1+Q_2 \end{bsmallmatrix} \prec 0$, recovering the result obtained in \cite{McCourtDissipSwitched} for DT dissipative switched systems. For a more complete review of linear interconnection models, the interested reader is referred to \cite{ArcakDissipativity}.
	Evidently, for each individual case, one needs to decouple \eqref{LMI global} by using the properties of each specific matrix $H$.
	
	\vspace{0.1cm}
	
\subsubsection*{Dual Approach} Next, we derive a formulation of the local conditions \eqref{LMIlocalprimal} in the dual supply rate matrices $(\mathscr{Q}, \mathscr{S},\mathscr{R})$.
	
\begin{prop} \label{prop:dualLMI}
	The inequality \eqref{LMI global} is satisfied with $Q\prec0$ and $R\succ0$ if and only if 
	\begin{equation} \label{LMIglobaldual}
		{\modif \begin{bmatrix} -H^\top \\ I	\end{bmatrix}\' \begin{bmatrix} \mathscr{Q} & \mathscr{S} \\ \star & \mathscr{R}
		\end{bmatrix} \begin{bmatrix} -H^\top \\ I	\end{bmatrix} \succ 0 }
	\end{equation}
holds with $\mathscr{Q}\prec 0$ and $\mathscr{R} \succ 0$,  where $\begin{bsmallmatrix} \mathscr{Q} & \mathscr{S} \\ \star & \mathscr{R}
\end{bsmallmatrix} = \begin{bsmallmatrix} Q & S \\ \star & R
\end{bsmallmatrix}^{-1}$.
\end{prop}

\begin{proof}
	{\modif The result is obtained by applying the Dualization Lemma (Fact \ref{Fact:Duality} in Appendix \ref{appendix:matrix}) to \eqref{LMI global} and recalling that $Q\prec 0, \; R\succ 0 \iff \mathscr{Q}\prec 0, \; \mathscr{R}\succ 0$. }
\end{proof}

Because of block matrix inversion properties \cite{BlockMatrixInversion}, we note that the matrices $\mathscr{Q}$, $\mathscr{S}$, and $\mathscr{R}$ in \eqref{LMIglobaldual} retain the same block diagonal structure of $Q$, $S$, and $R$ in \eqref{LMI global}. Thus by setting $\mathscr{Q} = \mbox{diag}(\mathscr{Q}_1,\ldots,\mathscr{Q}_N)$, $\mathscr{S} = \mbox{diag}(\mathscr{S}_1,\ldots,\mathscr{S}_N)$, and $\mathscr{R} = \mbox{diag}(\mathscr{R}_1,\ldots,\mathscr{R}_N)$, we observe that $\mathscr{Q}_i$, $\mathscr{S}_i$, and $\mathscr{R}_i$ are the dual matrices associated to system $i\in\mathcal{V}$. 

\begin{cor} \label{cor:LMI}
		The LMI \eqref{LMIglobaldual} is feasible under Laplacian coupling if at least one of the following decentralized LMIs is satisfied for all $i \in \mathcal{V}$,
	\begin{subequations} \label{LMIlocaldual}
		\begin{align}
			\mathscr{S}_i=\tfrac{1}{2}\alpha I, \; -\tfrac{1}{2d_i}I \prec \mathscr{Q}_i \prec 0, \; \mathscr{R}_i \succ 2  d_i \tilde{\alpha} I \,, \label{Thm2:first} \\ 
			\mathscr{S}_i=\mathcal{S} \succeq 0, \; -\tfrac{1}{2d_i}I \prec \mathscr{Q}_i \prec 0, \; \mathscr{R}_i \succ 2d_iI \,, \label{Thm2:second} \\ 
			0 \preceq \mathscr{S}_i \prec \tfrac{1}{3d_i}I, \; -\tfrac{1}{2d_i}I \prec \mathscr{Q}_i \prec 0, \; \mathscr{R}_i - \mathscr{S}_i \succ 4d_iI \,, \label{Thm2:third} \\ 
			\mathscr{S}_i\succeq 0, \; \mathscr{Q}_i-\mathscr{S}_i \succ -\tfrac{1}{2d_i}I, \; \mathscr{R}_i \succ 2\mathscr{S}_i, \; \mathscr{R}_i \succ 4d_iI \,. \label{Thm2:fourth}
		\end{align}
	\end{subequations}
\end{cor}

\vspace{0.2cm}

\begin{proof}
	Conditions \eqref{LMIlocaldual} follow from \eqref{LMIlocalprimal} by considering the change of variables $(Q_i,S_i,R_i)=(-\mathscr{R}_i, \mathscr{S}_i^\top, \mathscr{-Q}_i)$.
\end{proof}

Thanks to this result, each {\modif individual} subsystem can pair LMIs \eqref{LMIdual} and \eqref{LMIlocaldual} to design dissipative controllers and guarantee interconnection stability at the same time. We point out that the assumptions $Q\prec0$ and $R\succ0$ in Proposition \ref{prop:dualLMI} are the same appearing in Lemma \ref{lem:dissipativecontroldual}, and that the dual conditions \eqref{LMIlocaldual} are equivalent to \eqref{LMIlocalprimal} up to a change of variables. Therefore, recalling that $R\succeq0$ is a necessary condition for $(Q,S,R)$-dissipativity, if $Q$, $S$, and $R$ are fixed then the primal LMIs \eqref{LMIprimal}, \eqref{LMIlocalprimal} are equivalent to the dual ones \eqref{LMIdual}, \eqref{LMIlocaldual} provided that $R$ is non-singular. However, the main advantage of the dual LMIs is that they are linear in $\mathscr{Q}$, $\mathscr{S}$, and $\mathscr{R}$, hence we can efficiently treat them as decision variables.

\vspace{0.1cm}

\subsubsection*{Continuous-Time Systems}
For comparison purposes, consider a set of $N$ control-affine CT subsystems 
\begin{subequations}\label{subsystemsCT}
	\begin{align}
		\dot{x}_i & = f_i(x_i) + g_i(x_i)u_i \\
		y_i & = h_i(x_i) \,,
	\end{align}
\end{subequations}
with $f_i(0)=0$, $h_i(0)=0$ for all $i \in \mathcal{V}$, interconnected to each other via Laplacian coupling \eqref{coupling}. 

\begin{fact} \label{fact:chopra}
	\cite[Thm. 2.1]{ChopraPassivity}
	If subsystems \eqref{subsystemsCT} are passive with PD storage functions $V_i : \mathbb{R}^{n} \to \mathbb{R}_{>0}$, then the origin of the network is stable.
\end{fact}

\iffalse
The following corollary can be introduced, whose proof is along the same lines of Lemma \ref{Lem passivity linear} and \ref{lem:stabilitynonlinar} and thus omitted. 

\begin{cor} \label{cor:stabilityCT}
	If, in addition to the assumptions in Fact \ref{fact:chopra}, the subsystems are zero-state observable, then the origin of the global dynamics \eqref{global dynamics CT} is globally asymptotically stable.
\end{cor}
\fi

Similarly to the proofs of Lemma \ref{Lem passivity linear} and Theorem \ref{lem:stabilitynonlinar}, the proof of Fact \ref{fact:chopra} relies on a global Lyapunov function which is the sum of the {\modif individual} storage functions. We note that a similar proof can be constructed for the DT subsystems \eqref{subsystemsaffine}; that is, if subsystems \eqref{subsystemsaffine} are passive with PD storage functions, then the origin of the network \eqref{global dynamics} is stable. However, the assumption is vacuous since the DT subsystems without feedthrough \eqref{subsystemsaffine} cannot be passive. To conclude, we point out that Theorem \ref{lem:stabilitynonlinar} can be regarded as a DT counterpart of the CT results obtained in \cite{PotaLocalDissipativity}.

%%%%%%%%%%%%%%%%%%%%%%%%%%%%%%%%%%%%%%%%%%%%%%%%%%%%%%%%%

\section{Application to Microgrids Control} \label{sec:microgrids}

\begin{figure}[]
	\centering
	\begin{circuitikz}[scale=0.55]
		\ctikzset{bipoles/length=1cm}
		\draw (1,0) to[V, v=$V_{in}$, invert] (1,3);
		\draw (1,3) to[R,l=$R$] (3.5,3);
		\draw (3.5,3) to[L,i>^=$I$,l=$L$,] (6,3);
		\draw (6,0) to[C,l=$C$] (6,3);
		\draw (6,0) -- (2,0);
		\draw (2,0) -- (1,0); 
		\draw (6,3) -- (8.2,3) to[american current source,l=$I_{L}$] (8.2,0) -- (6,0);
		%\draw[-triangle 45] (7.1,0.2) -- (7.1,2.8) node[right, midway] {$V(t)$};
		\draw (8.2,3) to[short,i<=$I_{G}$] (11,3) to[short,-o] (11,3) -- (12.2,3);
		\draw (8.2,0) to[short,-o] (11,0) -- (12.2,0);
		\node at (11,2.5) {$+$};
		\node at (11,0.5) {$-$};
		\node at (11,1.5) {$V$};
		%\draw [draw=black] (12,-0.5) rectangle node[] {grid} (14,3.5);
		\node [draw, cloud, cloud puffs=15, aspect=2, cloud puff arc=120, line width=0.5, rotate=270, fill=blue!10, minimum width=2cm, minimum height=1cm] at (13,1.5) {microgrid};
	\end{circuitikz}
	\caption{Model of a distributed generation unit (DGU), comprising an input voltage, a Buck converter and a local load, connected to the rest of the microgrid.}
	\label{Fig:Buck}
\end{figure}
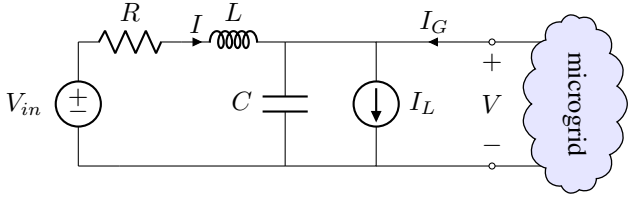

The increasing penetration of renewable resources and recent advances in power electronics have fuelled an increasing interest towards smart and flexible energy distribution systems such as microgrids. A microgrid usually comprises a set of spatially distributed subsystems -- called distributed generation units (DGUs) -- interconnected to each other via transmission lines.  These systems can operate either attached to the main grid or, as we consider in this example, in islanded mode.

\subsection{Microgrid Model}

\begin{figure}
	\centering
	\includegraphics[width=\columnwidth,trim={0 1.5cm 0 0.8cm},clip]{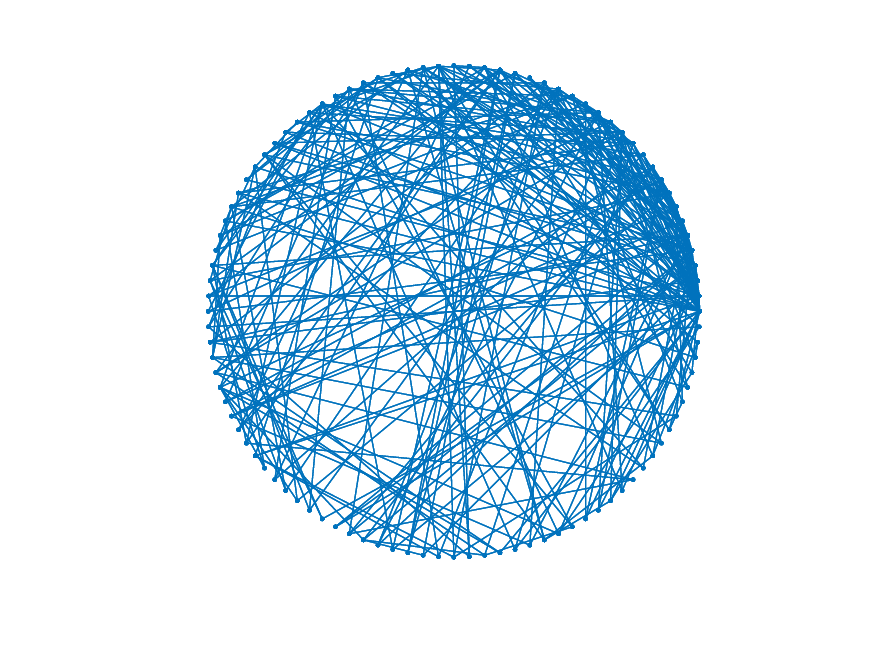}
	\caption{Graph with $N=100$ nodes representing the microgrid model, randomly generated by using a preferential attachment mechanism. }
	\label{fig:network}
\end{figure}

\begin{table}
	\centering
	\begin{tabular}{lll}
		\thickhline \\[-0.2cm]
		Parameter & Symbol & Value \\ \\[-0.25cm]
		\hline \\ [-0.2cm]
		Internal Resistance & $R_i$ & $0.2 \pm 0.05 \; \Omega$ \\
		Internal Inductance & $L_i$ & $2.5 \pm 1.0 \; mF$ \\
		Internal Capacitance & $C_i$ & $0.01 \pm 0.001 \; H$ \\
		Load Conductance & $Y_i$ & $0.02 \pm 0.001 \; S$ \\
		Line Resistance & $R_{ij}$ & $0.05 \; \Omega$ \\ \\ [-0.25cm]
		\thickhline
	\end{tabular}
	\caption{Electrical parameters used in the numerical experiments. The values for each DGU are drawn from uniform distributions indicated by the corresponding intervals.}
	\label{Tab:electricalparameters}
\end{table}

Consider the model of a DGU depicted in Figure \ref{Fig:Buck}. The input voltage $V_{in}$ represents a DC renewable source, \textit{e.g.}, a solar panel, and is designed to sustain an output voltage $V$ over a local load $I_L$, for instance a household or an electric vehicle. The $RLC$ circuit in between, with $R,L,C >0$, is the averaged model of a Buck converter, which steps down the input voltage so that $V \le V_{in}$ at all times. Note that, from a physical point of view, the DGU is regulated by acting on the duty cycle of the Buck converter which, in turn, modulates the input voltage $V_{in}$. For simplicity, we consider here manipulating the input voltage directly, disregarding the duty cycle saturation. The internal current is denoted with $I$, while the external current flowing from the neighbouring DGUs is $I_G$. We assume constant impedance loads $I_L = YV$, with conductance $Y>0$. The {\modif individual} dynamics of the \textit{i-th} DGU can then be represented by
\begin{subequations}\label{BuckDynamics}
	\begin{align}
		\dot{x}_i & = A_ix_i + B_iv_i + G_iu_i \\
		y_i & = C_ix_i \,,
	\end{align}
\end{subequations}
where $x_i = \begin{bmatrix} V_i & I_i \end{bmatrix}^\top$ is the state, $v_i = V_{in,i}$ is the local control input, $u_i=I_{G,i}$ is the current injected from the microgrid, and the matrices
\begin{align*}
	A_i & = \begin{bmatrix} -Y_i/C_i & 1/C_i \\ -1/L_i & -R_i/L_i  \end{bmatrix} \,, \quad B_i = \begin{bmatrix} 0 \\ 1/L_i \end{bmatrix} \,, \\
	G_i & = \begin{bmatrix} 1/C_i \\ 0 \end{bmatrix} \,, \quad C_i = \begin{bmatrix} 1 & 0 \end{bmatrix}\,,
\end{align*} 
depend on the electrical parameters of the converter. The subsystems are interconnected to each other via resistive lines, 
\begin{equation} \label{interconnectionlines}
	u_i = \sum_{j\in\mathcal{N}_i} \frac{1}{R_{ij}} (y_j-y_i) \,,
\end{equation}
where $R_{ij}>0$ denotes the resistance of the line connecting DGU $i$ with DGU $j$. Interconnection \eqref{interconnectionlines} is a Laplacian coupling (see \eqref{coupling}), hence can be represented as $u=-Ly$, where $u=\begin{bmatrix} u_1 & \cdots & u_N
\end{bmatrix}$ and $y=\begin{bmatrix} y_1 & \cdots & y_N
\end{bmatrix}$. By considering local feedback controllers $v_i = K_i x_i$, the overall network dynamics becomes
\begin{equation} \label{microgridglobal}
	\dot{x} = (A+BK-GLC)x \,,
\end{equation}
where $x=\begin{bmatrix} x_1 & \cdots & x_N \end{bmatrix}$, $A=\text{diag}(A_1,\ldots,A_N)$, $B=\text{diag}(B_1,\ldots,B_N)$, $K=\text{diag}(K_1,\ldots,K_N)$, $G=\text{diag}(G_1,\ldots,G_N)$, and $C=\text{diag}(C_1,\ldots,C_N)$. 

We consider a network of $N=100$ DGUs \eqref{BuckDynamics} interconnected via resistive lines  \eqref{interconnectionlines}. The interconnection topology is depicted in Fig. \ref{fig:network} and is randomly generated according to the Barabási-Albert model \cite{BarbasiAlbertModel} to capture the fact that power distribution networks are well represented by graphs constructed via preferential attachments; as a result, a restricted number of nodes (hubs) exhibit significant higher degree then the majority of other nodes. The electrical parameters for each DGU are uniformly selected in the intervals indicated in Table~\ref{Tab:electricalparameters}, centred at values similar to those used in the microgrid control literature, \textit{e.g.}, \cite{MicheleTCST}.

\begin{figure}
	\centering
	\begin{subfigure}{0.49\columnwidth}
		\includegraphics[width=\columnwidth,trim=1cm 0 2.1cm 0]{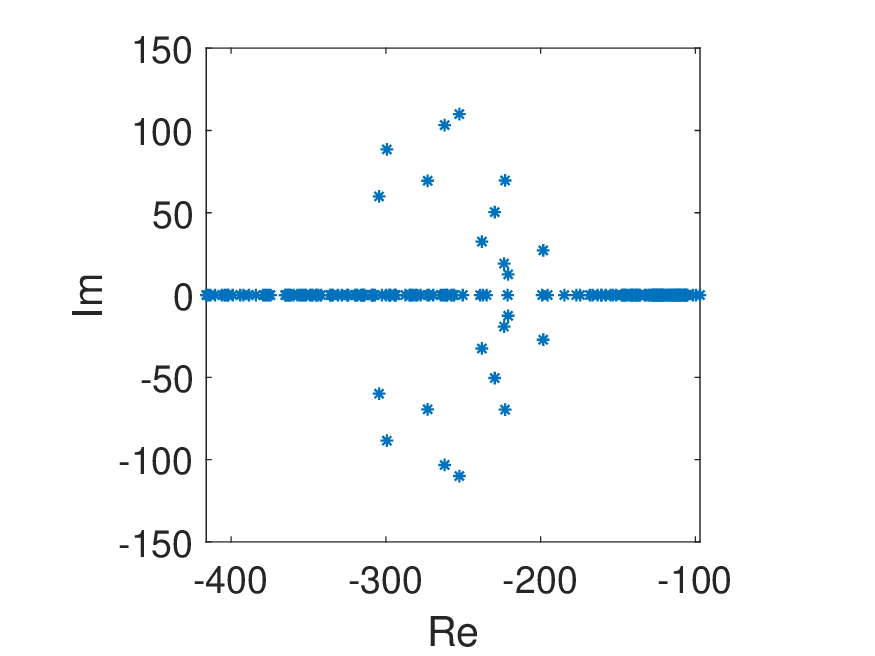}
		\caption{}
		\label{fig:eigCT}
	\end{subfigure}
	\begin{subfigure}{0.49\columnwidth}
		\includegraphics[width=\columnwidth,trim=1cm 0 2.1cm 0]{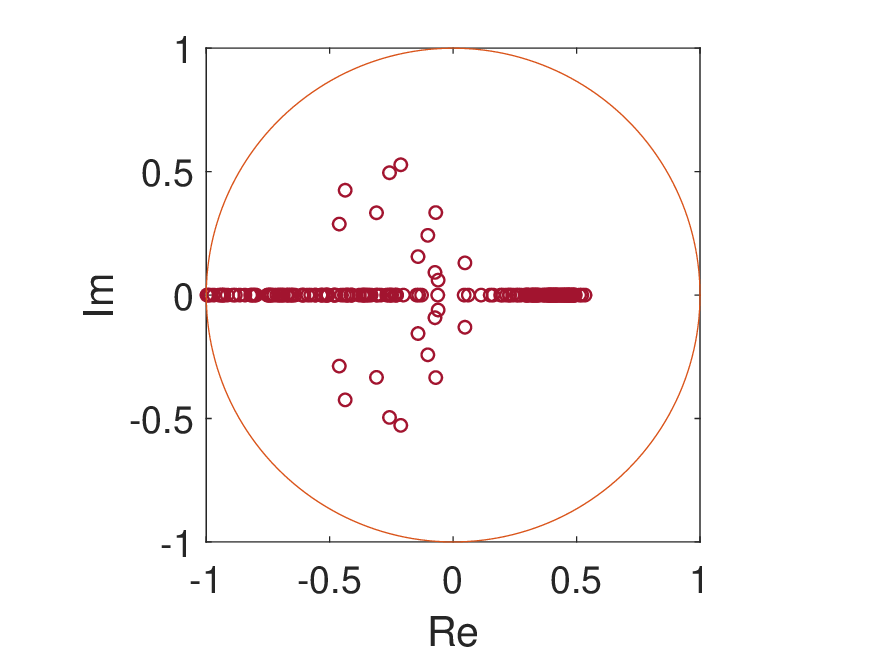}
		\caption{}
		\label{fig:eigDT}
	\end{subfigure}
	\begin{subfigure}{0.65\columnwidth}
		\hspace{-0.5cm}\includegraphics[width=\columnwidth,trim=1cm 0 2.1cm 0]{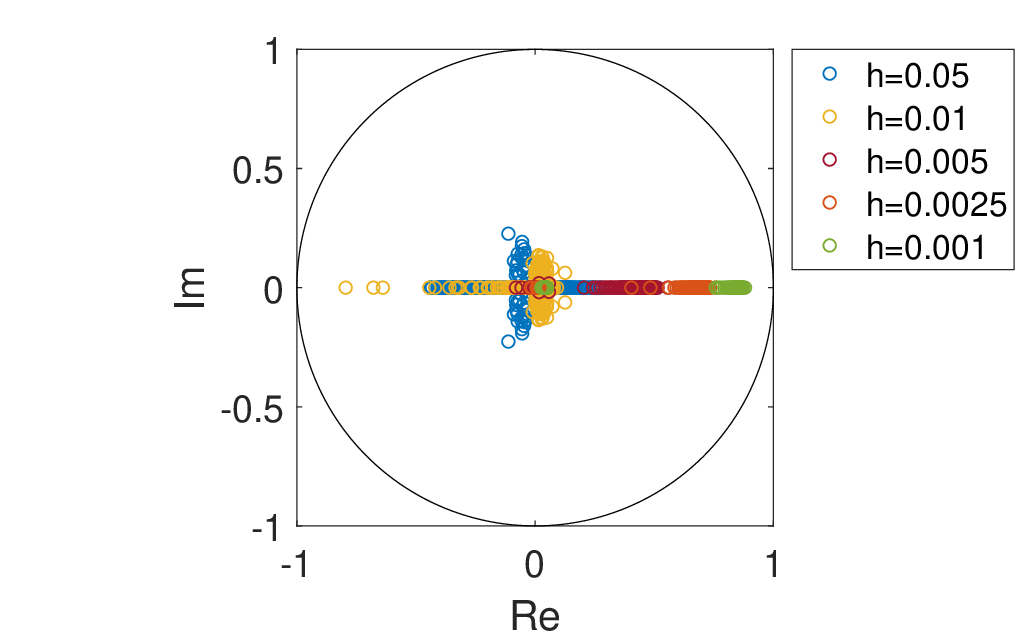} 
		\caption{}
		\label{fig:eigDTT}
	\end{subfigure}
	\caption{Eigenvalues distribution of (a) the CT system \eqref{microgridglobal}, (b) the corresponding DT system discretized with forward Euler and stepsize $h=0.005$, and (c) the DT system equipped with our dissipative controllers and for different discretization stepsizes $h$.}
	\label{fig:eig_distr}
\end{figure}

\iffalse
 Indeed, let us define the local PD storage function $E_i(x_i) = \tfrac{1}{2}C_iV_i^2 + \tfrac{1}{2}L_iI_i^2$, which represents the total energy stored in the converter. Then we compute
%
\begin{align*}
	\dot{E}_i(x_i) & = C_iV_i\dot{V}_i + L_iI_i\dot{I}_i \\
	& = -Y_i V_i^2 - (R_i+K_{I,i})I_i^2 + y_iu_i \\
	& \le y_iu_i \,,
\end{align*}
concluding that the subsystems are passive. Therefore, one can invoke Fact \ref{fact:chopra} to infer that the network \eqref{microgridglobal} is stable. By taking a global Lyapunov function $E(x)$ as the sum of the local storage functions $E_i(x_i)$, we can write
%
\begin{equation*}
	\dot{E}(x) = -YV^2 - (R+K_I)I^2 + y^\top u\,,
\end{equation*}
where $Y=\text{diag}(Y_1,\dots,Y_N)$, $R=\text{diag}(R_1,\dots,R_N)$, and $K_I=\text{diag}(K_{I,1},\dots,K_{I,N})$. By \cite{ChopraPassivity} for passive CT systems interconnected in a Laplacian fashion, $y^\top u \le 0$, therefore $\dot{E}(x)<0$ for all $x \ne 0$ and we conclude that \eqref{microgridglobal} is asymptotically stable.
\fi

\subsection{Discretization and Dissipative Control}

If the local controllers $K_i=\begin{bmatrix} K_{V,i} & K_{I,i} \end{bmatrix}$ are set to be proportional to the internal current, \textit{i.e.}, $K_{V,i}=0$ and $K_{I,i}<0$, then subsystems \eqref{BuckDynamics} become passive \cite{PulkitAutomatica}. Then, by Fact \ref{fact:chopra} and CT stability arguments, one can show that \eqref{microgridglobal} is asymptotically stable. To confirm this, in Fig. \ref{fig:eigCT} we show the eigenvalues distribution of \eqref{microgridglobal} when the local controllers are drawn from a uniform distribution $K_{I,i}= -1 \pm 0.1$.

Unfortunately, as already discussed, properties like passivity or stability can be lost under discretization \cite{BrogliatoDissipativity,StramigioliDiscretization, LailaDissipationSampling, CostaPassivityPreservation, OishiPassivityDegradation}, and most of the available techniques compromise the sparsity pattern of the matrices involved \cite{FarinaDiscretization, SouzaDiscretization}. For instance, forward Euler preserves stability of \eqref{microgridglobal} if and only if the stepsize $h< h^* \coloneq \min_{i}\tfrac{-2\text{Re}(\lambda_i)}{| \lambda_i |^{2}}$, where the $\lambda_i$'s are the eigenvalues of the CT system \cite{FarinaDiscretization}. In this example $h^* \approx 0.005s$, as confirmed by Fig. \ref{fig:eigDT}. In case of nonlinear systems, fast dynamics, or different discretization methods, a suitable stepsize might be extremely small or difficult to estimate. For these reasons, we want to implement our distributed control methods directly on the DT model. 

\begin{figure}
	\centering
	\includegraphics[width=0.48\columnwidth,trim=1.3cm 0 1.7cm 0]{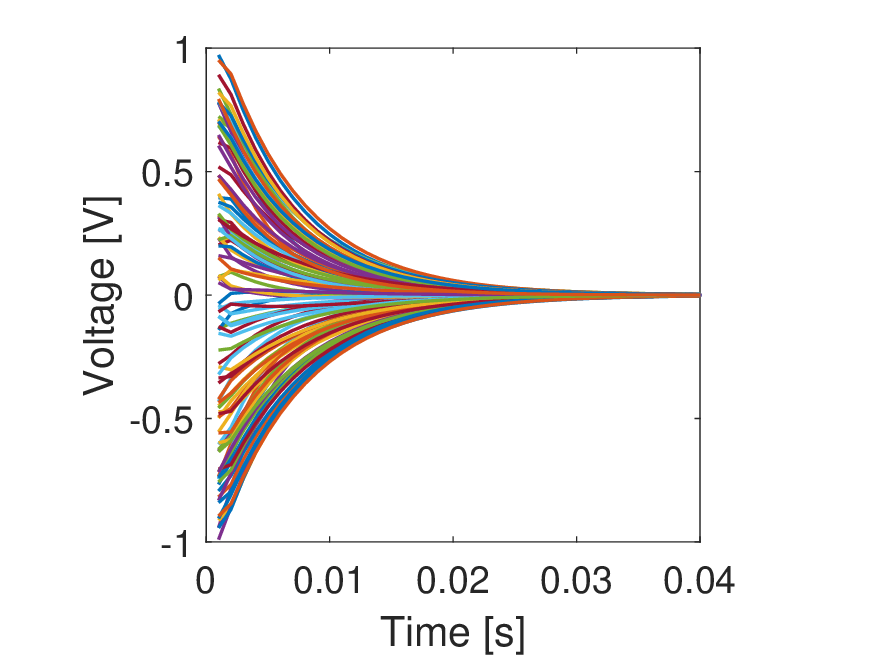} \hspace{0.1cm}
	\includegraphics[width=0.48\columnwidth,trim=1.3cm 0 1.7cm 0]{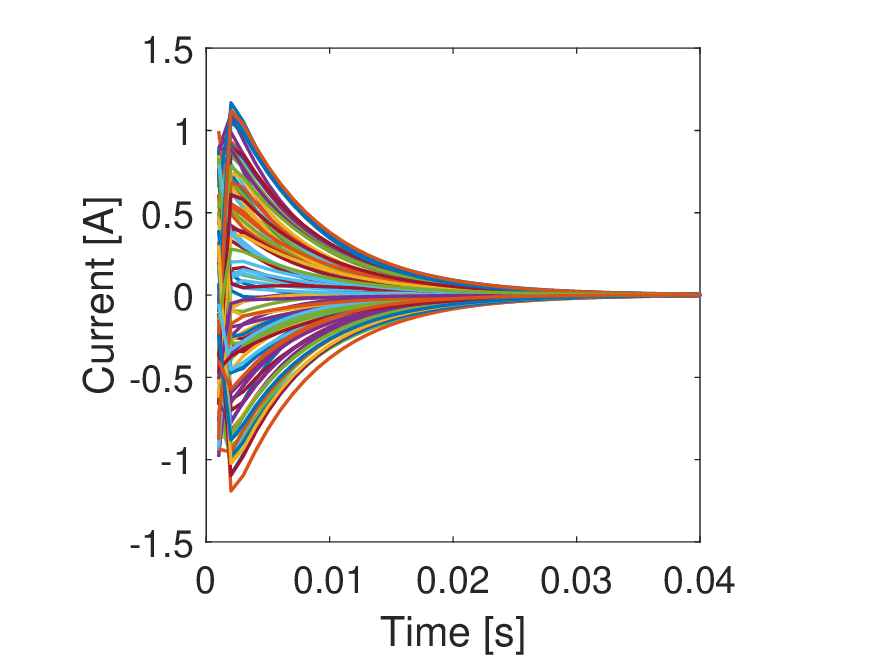}
	\includegraphics[width=0.48\columnwidth,trim=1.3cm 0 1.7cm 0]{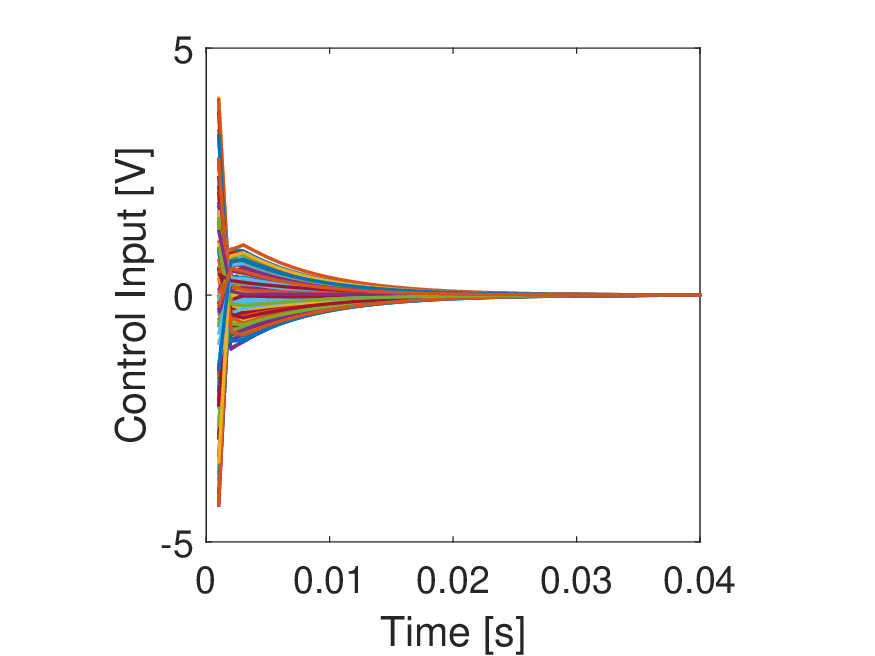}
	\hspace{0.1cm}
	\includegraphics[width=0.48\columnwidth,trim=1.3cm 0 1.7cm 0]{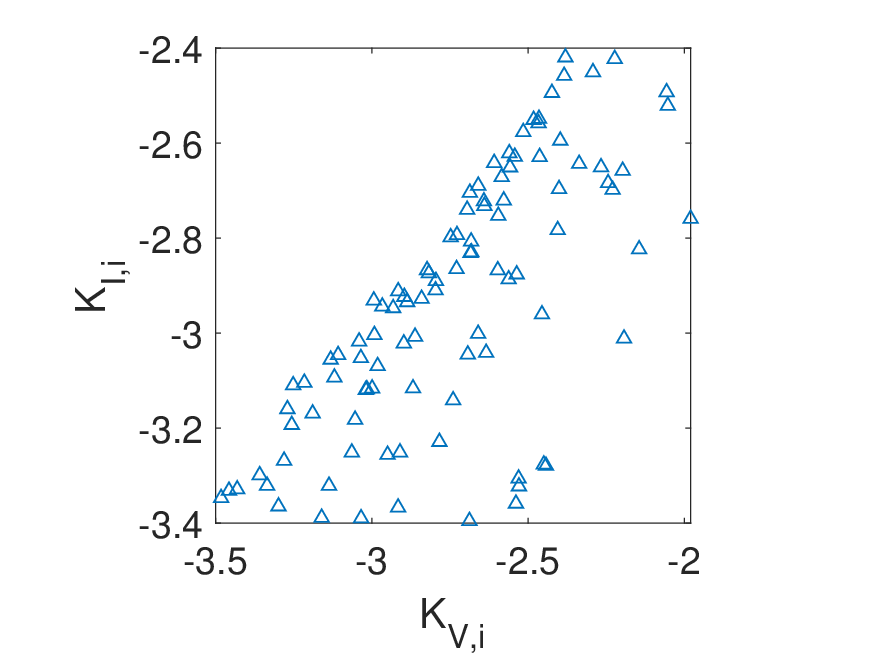}
	\caption{The local dissipative controllers steer the DT network trajectories back to the origin after a perturbation of the output voltages drawn uniformly from the interval $[-1V, 1V]$. Note that for simplicity we consider the problem of stabilization to the origin but, in general, stabilization to a desired non-zero state is achieved by a suitable change of variables.}
	\label{fig:simulation}
\end{figure}

\begin{figure}
	\centering
	\includegraphics[width=0.4\textwidth]{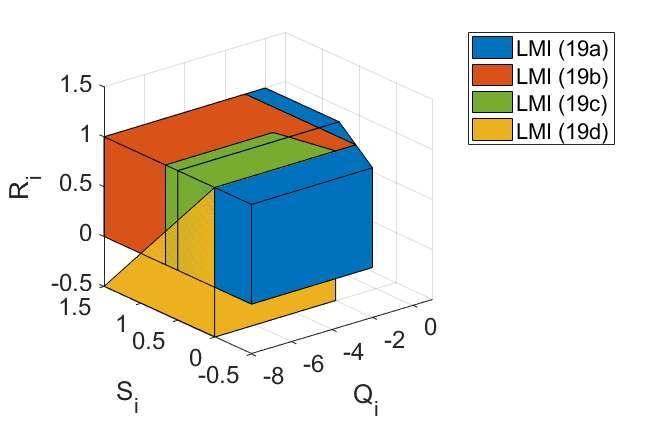}
	\caption{A portion of the feasible region of the four LMIs in Theorem \ref{TheoremBounds} for a node with degree $d_i=0.5$.}
	\label{fig:LMIplot}
\end{figure}

Instead of designing controllers for the CT system, we first discretize each {\modif individual} system \eqref{BuckDynamics} by sampling and holding the coupling variables $u_i$ with different stepsizes, as shown in Fig.~\ref{fig:eigDTT}. Then, each DT system solves LMI \eqref{LMIdual} to compute a dissipative controller and any of the LMIs \eqref{LMIlocaldual} to guarantee stability of the interconnection in the dual variables $(P_i,Z_i,\mathscr{Q}_i,\mathscr{S}_i,\mathscr{R}_i)$. In case of feasibility, the local dissipative controllers are given by $K_i=Z_iP_i^{-1}$ and they guarantee asymptotic stability of the DT network. For this specific microgrid example, we observe that \eqref{Thm2:first} and \eqref{Thm2:second} are feasible for a wide choice of parameters $\alpha$ and $\mathcal{S}$, while \eqref{Thm2:third} and \eqref{Thm2:fourth} are always feasible. We depict in Fig.~\ref{fig:eigDTT} the eigenvalue distribution of the resulting DT system equipped with the local dissipative controllers computed via \eqref{LMIdual} and \eqref{Thm2:first} with the arbitrary choice of $\alpha=1$. In Fig.~\ref{fig:simulation} we report the values and performance of these controllers when $h=0.001$. We mention that the problem of selecting specific controllers within a set of dissipative ones is still open. Although they all stabilize the network, they might have different performance in terms of transients, frequency behaviour or robustness. Some heuristic approaches include maximising the dissipation rate or minimizing the distance between $P_i$ and the solution to the linear quadratic regulator problem \cite{AhmedPassivity}.

At last, we wish to compare the four local LMIs in Theorem~\ref{TheoremBounds}, that we recall to be equivalent to the dual LMIs in Corollary \ref{cor:LMI} up to a change of variables. For single-input single-output systems ($m=1$), like in this example, the variables $Q_i$, $S_i$, and $R_i$ are scalar and we can conveniently represent the sets in 3D, as shown in Fig. \ref{fig:LMIplot}. Note that \eqref{Thm1:first} and \eqref{Thm1:second} are plotted by treating the global parameters $\alpha$ and $\mathcal{S}$ as decision variables (in a higher-dimensional space), and then projecting the set in 3D. In this way we can appreciate the set geometry for a continuous range of $\alpha$ and $\mathcal{S}$ but, in practice, they have to be fixed in advance and be the same for each system in the network. For a fixed choice of $\alpha$ or $\mathcal{S}$, the corresponding LMIs \eqref{Thm1:first} and \eqref{Thm1:second} are found by intersecting the depicted volumes with the planes $S_i= \tfrac{1}{2}\alpha I$ and $S_i=\mathcal{S}$, respectively. As an example, the LMI \eqref{Thm1:first} with $\alpha=1$ is depicted in Fig.~\ref{fig:LMIplot} as the intersection of the blue region with the plane $S_i=0.5$.

An important observation is that, even in the most simple scenario of scalar variables, the intersection geometry of the four sets is not trivial and there is no dominant condition. Although \eqref{Thm1:first} and \eqref{Thm1:second} seem to subsume \eqref{Thm1:third}, once $\alpha$ or $\mathcal{S}$ are fixed they reduce to lower-dimensional sets; hence, their performance depends on the ability to select appropriate global parameters. Moreover, even if \eqref{Thm1:second} is contained in \eqref{Thm1:first} in this one-dimensional example, this is not true anymore for $m>1$ since $\mathcal{S}\in\mathbb{R}^{m\times m}$ but $\alpha\in\mathbb{R}$.

%%%%%%%%%%%%%%%%%%%%%%%%%%%%%%%%%%%%%%%%%%%%%%%%%%%%%%%%%%%%%%%%%%%%%%%%

\section{Conclusion}

We shed light on the problem of stability of interconnected DT {\modif dissipative} systems without feedthrough. We now understand virtual outputs as a special case of $(Q,S,R)$-dissipativity, and realize that $R$ must be PSD but not zero. Moreover, we derived the stability condition \eqref{LMI global} for nonlinear interconnected $(Q,S,R)$-dissipative systems, and we can use LMIs \eqref{LMIprimal} and \eqref{LMIlocalprimal} (or \eqref{LMIdual} and \eqref{LMIlocaldual} in the dual variables) to design local dissipative controllers that guarantee stability of the {\modif entire} network.

Numerous exciting research directions remain open, such as how to compute dissipative controllers for nonlinear systems with inequalities that are linear in the supply rate, how to deal with stochastic systems, unknown dynamics, time-varying or uncertain topologies, deriving more general sufficient conditions \eqref{LMIlocalprimal} {\modif or relaxing the Laplacian assumption in Theorem~\ref{TheoremBounds}.} {\modif Finally, we wish to extend the results of this paper beyond the class of quadratic supply rate functions.}

\section*{Acknowledgements}

We would like to thank J. Eising for the stimulating discussions on dissipativity theory.

%%%%%%%%%%%%%%%%%%%%%%%%%%%%%%%%%%%%%%%%%%%%%%%%%%%%%%%%%%%%%%%%%%%%%%%%

\bibliographystyle{ieeetr}
\bibliography{Bibliography}

%%%%%%%%%%%%%%%%%%%%%%%%%%%%%%%%%%%%%%%%%%%%%%%%%%%%%%%%%%%%%%%%%%%%%%%%

\begin{appendices}
	
\section{Matrix Theory} \label{appendix:matrix}

 We introduce a set of facts regarding PSD matrices that are extensively used throughout the manuscript. 

\begin{fact}\label{Fact:congruence} \cite[Obs. 7.1.8]{HornJohnsonMatrix}
	Consider $A\in\mathbb{R}^{n\times n}$ and $B\in\mathbb{R}^{n\times m}$. If $A \succeq 0$, then $B\' AB \succeq 0$. If $A\succ 0$, then $B\' AB \succ 0$ if and only if $\text{Rank}(B) = m$.
\end{fact}

\iffalse
\begin{fact}\label{Fact:changesign}
	Given the real square matrices $A$, $B$ and $C$, the following relation holds,
	\begin{equation*}
		\begin{bmatrix}
			A & B \\ \star & C
		\end{bmatrix} \succeq 0 \; (\mbox{resp.} \succ 0) \iff \begin{bmatrix}
			A & -B \\ \star & C
		\end{bmatrix} \succeq 0 \; (\mbox{resp.} \succ 0)\,,
	\end{equation*}
	where $\star$ denotes symmetry.
\end{fact}

\begin{proof}
	Assume that $\begin{bsmallmatrix} A & B \\ \star & C 
	\end{bsmallmatrix} \succeq 0$. Then, for all vectors $x$ and $y$ of consistent dimension, it holds
	%
	\begin{align*}
		\begin{bmatrix}
			x \\ y
		\end{bmatrix}^\top \begin{bmatrix}
			A & -B \\ \star & C
		\end{bmatrix} \begin{bmatrix}
			x \\ y
		\end{bmatrix} & = x^\top Ax - y^\top B^\top x - x^\top B y + y^\top C y \\
		& = \begin{bmatrix}
			x \\ -y
		\end{bmatrix}^\top \begin{bmatrix}
			A & B \\ \star & C
		\end{bmatrix} \begin{bmatrix}
			x \\ -y
		\end{bmatrix} \succeq 0 \,,
	\end{align*}
	and vice versa the reasoning holds starting from the assumption that $\begin{bsmallmatrix} A & -B \\ \star & C 
	\end{bsmallmatrix} \succeq 0$. A similar discussion is verified for the PD case.
\end{proof}
\fi

\vspace{0.1cm}

\begin{fact}[Dualization Lemma]\label{Fact:Duality} \cite[Cor. 4.10]{SchererLMIs}
	Consider $A\in\mathbb{R}^{n \times n}$, $B\in\mathbb{R}^{n \times p}$, $C\in\mathbb{R}^{p\times p}$, and $M\in\mathbb{R}^{p \times n}$. If $\begin{bsmallmatrix} A & B \\ \star & C \end{bsmallmatrix}$ is invertible and its inverse is $\begin{bsmallmatrix} \mathscr{A} & \mathscr{B} \\ \star & \mathscr{C} \end{bsmallmatrix}$, then
	\begin{align*}
		\begin{bmatrix}
			I \\ M
		\end{bmatrix}^\top \begin{bmatrix} A & B \\ \star & C \end{bmatrix} \begin{bmatrix}
			I \\ M
		\end{bmatrix} \prec 0 \,, \quad C \succeq 0 \quad \iff \\
		\begin{bmatrix}
			-M^\top \\ I
		\end{bmatrix}^\top \begin{bmatrix} \mathscr{A} & \mathscr{B} \\ \star & \mathscr{C} \end{bmatrix} \begin{bmatrix}
			-M^\top \\ I
		\end{bmatrix} \succ 0 \,, \quad \mathscr{A} \preceq 0\,.
	\end{align*}
	Moreover, if $\text{In}( \begin{bsmallmatrix} A & B \\ \star & C \end{bsmallmatrix}) = (p,0,n)$, then the previous relation holds with non-strict inequalities \cite[Lemma 2]{HenkDissipativity}.
\end{fact}

\begin{fact}\label{Fact:psd}
	Given the real square matrices $A$, $B$ and $C$, the following relations hold,
	\begin{equation*}
		\begin{aligned}
			A \succeq B \,, \;\, C \succeq B \succeq 0 \implies \begin{bmatrix}
				A & B \\ \star & C
			\end{bmatrix} \succeq 0 \,, \\
			A \succ B \,, \;\, C \succ B \succeq 0 \implies \begin{bmatrix}
				A & B \\ \star & C
			\end{bmatrix} \succ 0 \,.
		\end{aligned}
	\end{equation*}
\end{fact}

\vspace{0.1cm}

\begin{proof}
	We prove here the first relation, while the second holds analogously. Consider the following decomposition,
	\begin{equation*}
		\begin{bmatrix}
			A & B \\ \star & C
		\end{bmatrix} = \begin{bmatrix}
			B & B \\ \star & B
		\end{bmatrix} + \begin{bmatrix}
			A-B & 0 \\ \star & C-B
		\end{bmatrix} \succeq 0 \,.
	\end{equation*}
	The first matrix in the decomposition can be written as the Kronecker product of two PSD matrices as $\begin{bsmallmatrix}
		B & B \\ \star & B \end{bsmallmatrix} = \mathbf{1} \mathbf{1}^\top \otimes B$, hence it is PSD (see \cite[Cor. 4.2.13]{HornJohnsonMatrixTopics}). The second matrix is PSD since it is block diagonal with PSD blocks. 
\end{proof}

\section{Graph Theory} \label{appendix:graph}

 We introduce some important facts about algebraic graph theory, as well as novel results on Laplacian flows and bounds on the Laplacian matrix pseudoinverse.

\begin{fact}\label{fact:Laplacian}
	\cite[Ch. 6]{BulloNetwork}
	The Laplacian matrix of an undirected graph is symmetric, PSD, row-stochastic (\textit{i.e.}, $\mathcal{L} \mathbf{1} = 0$), its spectrum is real and can be ordered as $0 = \lambda_1 \le \lambda_2 \le \dots \le \lambda_N$. The second eigenvalue $\lambda_2$ is strictly greater than zero if and only if the graph is connected.
\end{fact} 

In this work we only consider connected graphs ($d_{\text{min}}, \lambda_{2}>0$). An important differential equation that often arises in multi-agent systems is the so-called \textit{Laplacian flow}, $\dot{x}=-\mathcal{L}x$. 
\iffalse
Since all the eigenvalues of $-\mathcal{L}$ have non-positive real part, and only one eigenvalue is at zero, the Laplacian flow for connected graphs is stable (but not asymptotically stable). This property can also be explained by invoking Lyapunov theory and considering $V(x)=x^\top \mathcal{V} x$ as Lyapunov function, where $\mathcal{V}=\text{diag}(v)$ and $v$ is the eigenvector associated to the zero eigenvalue of $\mathcal{L}$. In the following, we report a simplified version of this result to reflect the fact that we only consider undirected graphs (and not digraphs), thus $p=\mathbf{1}$. 
\fi

\begin{fact}
	\cite{ZhangLaplacianFlow}
	Let $\mathcal{L}$ be the Laplacian matrix of an undirected graph. Then, $\mathcal{V}=I$ satisfies the Lyapunov inequality $\mathcal{V} \mathcal{L} + \mathcal{L} \mathcal{V}^\top \succeq 0$.
\end{fact}

Next, we generalize this result to extended Laplacian flows of the form $\dot{x}=-(\mathcal{L} \otimes I)x$.
Note that $\mathcal{L} \otimes I$ is still a Laplacian matrix, thus retains all the associated properties previously described. We show in the next result that any PSD block diagonal matrix with identical blocks defines a valid Lyapunov function for the extended Laplacian flow.

\begin{prop}\label{lem:laplacianflow}
	Let $L \coloneqq (\mathcal{L} \otimes I)$, where $\mathcal{L}$ is the Laplacian matrix of an undirected graph. Then, any matrix of the form $V=I \otimes \mathcal{V}$, with $\mathcal{V}\succeq 0$, satisfies the Lyapunov inequality $VL + L V^\top \succeq 0$.
\end{prop}

\begin{proof}
	We can write 
	 \begin{align*}
			VL + LV^\top & = (I \otimes \mathcal{V}) (\mathcal{L} \otimes I) + (\mathcal{L} \otimes I)(I \otimes \mathcal{V})^\top \\ 
			& = \mathcal{L} \otimes \mathcal{V} + \mathcal{L} \otimes \mathcal{V}^\top \\
			& = \mathcal{L} \otimes (\mathcal{V} + \mathcal{V}^\top) \succeq 0 \,,
	\end{align*}
	where in the second and third equality we used the mixed-product property \cite[Lemma 4.2.10]{HornJohnsonMatrixTopics} and the associative property \cite[Prop. 4.2.8]{HornJohnsonMatrixTopics} of the Kronecker product, respectively. The last inequality follows from the fact that the Kronecker product of PSD matrices is a PSD matrix \cite[Cor. 4.2.13]{HornJohnsonMatrixTopics}. 
\end{proof}

Although $\mathcal{L}$ is not invertible, its pseudoinverse \cite[p. 453]{HornJohnsonMatrix} $\mathcal{L}\psd$ enjoys important properties.
\begin{fact} \label{fact:Laplacianpseudo}
	\cite[Lemma 6.12]{BulloNetwork}
	The Laplacian pseudoinverse is symmetric, PSD, row-stochastic (\textit{i.e.}, $\mathcal{L}\psd \mathbf{1} = 0$), and, moreover, $\mathcal{L}\psd \mathcal{L} = \mathcal{L} \mathcal{L}\psd = I - \tfrac{1}{N}\mathbf{1} \mathbf{1}\'$.
\end{fact} 
\begin{fact}\label{FactRegLaplacian} \cite[E6.11]{BulloNetwork}
	The \textit{regularized} Laplacian $\mathcal{L}+\tfrac{\beta}{N}\mathbf{1}\mathbf{1}\'$ is PD for any $\beta>0$, and its inverse is $\mathcal{L}\psd+\tfrac{1}{\beta N}\mathbf{1}\mathbf{1}\'$.
\end{fact}

\begin{lem}\label{Fact:Laplacianbound}
	Let $\mathcal{L}$, $\mathcal{A}$, and $\mathcal{D}$ be the Laplacian, adjacency, and degree matrix of an undirected connected graph. Then,
	\begin{subequations}
		\begin{align}
			\mathcal{L} & \preceq 2\mathcal{D},  \label{Fact:Laplacianbound1} \\
			\mathcal{L}\psd + \tfrac{1}{d_{\text{min}} N}\mathbf{1}\mathbf{1}\' & \succeq \tfrac{1}{3}\mathcal{D}^{-1}. \label{Fact:Laplacianbound2}
		\end{align}
	\end{subequations}
\end{lem}
\vspace{0.2cm}
\begin{proof}
	The first bound can be proved by considering that $2\mathcal{D}-\mathcal{L} = 2\mathcal{D}-(\mathcal{D} - \mathcal{A}) = \mathcal{D} + \mathcal{A}$. Since $\mathcal{D} + \mathcal{A}$ is symmetric, nonnegative, and diagonally dominant, we can invoke Geršgorin circle theorem \cite[Theorem 6.1.1]{HornJohnsonMatrix} to conclude that $\mathcal{D}+\mathcal{A}\succeq 0$ and hence $\mathcal{L}\preceq 2\mathcal{D}$.
	
	As for the second bound, we note that the spectrum of $\tfrac{d_{\text{min}}}{N}\mathbf{1}\mathbf{1}\'$ is $\{ d_{\text{min}}, 0, \ldots, 0 \}$. Then, since $\lambda_{\text{max}}(\tfrac{d_{\text{min}}}{ N}\mathbf{1}\mathbf{1}\') \le \lambda_{\text{min}}(\mathcal{D})$, one can conclude that $\tfrac{d_{\text{min}}}{N}\mathbf{1}\mathbf{1}\' \preceq \mathcal{D}$ as a consequence of Weyl's inequality \cite[Theorem 4.3.1]{HornJohnsonMatrix}. Combining with \eqref{Fact:Laplacianbound1} leads to $\mathcal{L} + \tfrac{d_{\text{min}}}{N}\mathbf{1}\mathbf{1}\' \preceq 3\mathcal{D}$. Since both sides of the inequality are PD matrices, we can invert the relation and verify via Fact \ref{FactRegLaplacian} that $\mathcal{L}\psd + \tfrac{1}{d_{\text{min}} N}\mathbf{1}\mathbf{1}\' \succeq \tfrac{1}{3}\mathcal{D}^{-1}$.
\end{proof}

\section{Proof of Theorem \ref{TheoremBounds}} \label{appendix:theorem}
	
Recall that we work under Laplacian coupling ($H={\modif -} L$) and, to facilitate readability, we highlight in \textcolor{magenta}{magenta} the final conditions that appear in \eqref{LMIlocalprimal}.
	
\paragraph*{\normalsize Proof of LMI \eqref{Thm1:first}}
	By setting \textcolor{magenta}{$S=\tfrac{1}{2}\alpha I$}, \eqref{LMI global} becomes $\alpha L - LRL - Q \succ 0$ which, for \textcolor{magenta}{$R\succ 0$}, holds if and only if
	\begin{equation}\label{LMI0}
		\begin{bmatrix}
			\alpha L - Q & L \\ \star & R\inv
		\end{bmatrix} \succ 0 \,.
	\end{equation}
	Invoking Fact \ref{Fact:psd}, we infer that \eqref{LMI0} holds if $Q \prec -(1-\alpha)L$ and $R\inv \succ L \succeq 0$. Note that $L \succeq 0$ holds because $L$ is a Laplacian matrix (see Fact \ref{fact:Laplacian}), and define the matrix $D \coloneq \mathcal{D} \otimes I$. Then, $R\inv \succ L$ is implied by $R\inv \succ 2D$ because of Lemma~\ref{Fact:Laplacianbound}, which in turn holds if and only if \textcolor{magenta}{$R \prec \tfrac{1}{2}D\inv$}. To decouple the condition $Q \prec -(1-\alpha)L$ we consider two separate cases:
	
	\begin{itemize}
		\item $\alpha<1$. Since $1-\alpha > 0$, one can exploit Lemma~\ref{Fact:Laplacianbound} to infer that $Q \prec -(1-\alpha)L$ is implied by $Q\prec -(1-\alpha)2D$.
		
		\item $\alpha \ge 1$. In this case, since $1-\alpha\le0$, by observing that $L\succeq 0$ one can can simply impose $Q\prec 0$.
	\end{itemize}
	
	By merging the previous two conditions, and recalling that $\tilde{\alpha} = \max \{ 1-\alpha, 0 \}$, we conclude that $Q \prec (\alpha-1)L$ is implied by \textcolor{magenta}{$Q \prec - 2 \tilde{\alpha} D$}. Finally, the LMI \eqref{Thm1:first} in Theorem \ref{TheoremBounds} follows by noting that $Q$, $S$, $R$, and $D$ are block diagonal matrices. \hfill \QED
	
	\vspace{0.1cm}
	
\paragraph*{\normalsize Proof of LMI \eqref{Thm1:second}}
	Note that \eqref{LMI global} is implied by 
	\begin{subequations}
		\begin{align}
			SL+ L S\' \succeq 0 \,, \label{first} \\
			-LRL - Q \succ 0 \,. \label{second}
		\end{align}
	\end{subequations}
	By noticing that \eqref{first} is a Lyapunov inequality for the extended Laplacian flow $\dot{x} = -Lx = -(\mathcal{L} \otimes I)x$, we can invoke Proposition \ref{lem:laplacianflow} to state that \eqref{first} holds whenever $S$ is a block diagonal matrix with identical PSD blocks, \textcolor{magenta}{$S=I\otimes \mathcal{S} \succeq 0$}.
	If \textcolor{magenta}{$R\succ0$}, by the Schur complement \eqref{second} holds if and only if 
	\begin{equation*}
		\begin{bmatrix}
			-Q & L \\ \star & R^{-1}
		\end{bmatrix} \succ 0 \,,
	\end{equation*}
	which in turn is verified if $Q\prec -L$ and $R^{-1}\succ L$ by Fact \ref{Fact:psd}. We already showed in the proof of LMI \eqref{Thm1:first} that the latter inequality is implied by \textcolor{magenta}{$R \prec \tfrac{1}{2}D\inv$}. Following a similar reasoning based on Lemma \ref{Fact:Laplacianbound}, we claim that the former inequality is implied by \textcolor{magenta}{$Q\prec -2D$}. \hfill \QED
	
	\vspace{0.1cm}
	
\paragraph*{\normalsize Proof of LMI \eqref{Thm1:third}}
	By the Schur complement, if \textcolor{magenta}{$R \succ 0$}, the LMI \eqref{LMI global} is feasible if and only if 
	\begin{equation}\label{LMIa}
		\begin{bmatrix}
			LS\' + SL - Q & L \\ \star & R\inv
		\end{bmatrix} \succ 0 \,.
	\end{equation}
	Invoking Fact \ref{Fact:psd}, we infer that \eqref{LMIa} holds if $LS\' + SL - Q \succ L$ and $R\inv \succ L \succeq 0$. As previously discussed, the latter condition is implied by \textcolor{magenta}{$R \prec \tfrac{1}{2}D\inv$}.
	On the other hand, recalling that $L L\psd L = L$ and $L \mathbf{1} = 0$, $LS\' + SL - Q \succ L$ can be written as
	\begin{equation*}
		\begin{bmatrix}
			I \\ L
		\end{bmatrix}\' \begin{bmatrix}
			-Q-2L & S \\ \star & L\psd +\tfrac{1}{d_{\emph{min}}N}\mathbf{1}\mathbf{1}\'
		\end{bmatrix} \begin{bmatrix}
			I \\ L
		\end{bmatrix} \succ 0 \,,
	\end{equation*}
	which, according to Fact \ref{Fact:congruence}, is feasible if
	\begin{equation}\label{LMIb}
		\begin{bmatrix}
			-Q-2L & S \\ \star & L\psd +\tfrac{1}{d_{\emph{min}}N}\mathbf{1}\mathbf{1}\'
		\end{bmatrix} \succ 0 \,.
	\end{equation}
	Thanks to Fact \ref{Fact:psd}, we know that \eqref{LMIb} holds if \textcolor{magenta}{$S\succeq 0$}, $Q+S \prec -2L$ and $S \prec L\psd +\tfrac{1}{d_{\emph{min}} N}\mathbf{1}\mathbf{1}\'$. Finally, according to Lemma~\ref{Fact:Laplacianbound}, the previous two conditions are implied by \textcolor{magenta}{$Q+S \prec -4D$} and \textcolor{magenta}{$S \prec \tfrac{1}{3}D^{\inv}$}, respectively. \hfill \QED
	
	\vspace{0.1cm}
	
	\paragraph*{\normalsize Proof of LMI \eqref{Thm1:fourth}} 
	We add and subtract $\tfrac{1}{2}L D^{-1} L$ to \eqref{LMI global} and obtain
	\begin{equation*}\label{LMI1}
		L S\' + SL - \tfrac{1}{2}Q - L (R - \tfrac{1}{2}D^{-1}) L - \tfrac{1}{2}Q - \tfrac{1}{2}L D^{-1} L \succ 0 \,.
	\end{equation*}	
	The latter condition is implied by 
	\begin{subequations}
		\begin{align}
			- \tfrac{1}{2}Q - \tfrac{1}{2}L D^{-1} L \succeq 0 \,, \label{secondLMI} \\
			L S\' + SL - \tfrac{1}{2}Q -L (R - \tfrac{1}{2}D^{-1}) L \succ 0 \,. \label{firstLMI}
		\end{align}
	\end{subequations}
	By the Schur complement, \eqref{secondLMI} holds if and only if 
	\begin{equation}
		\begin{bmatrix}
			-\tfrac{1}{2}Q & L \\ \star & 2D
		\end{bmatrix} \succeq 0,
	\end{equation}
	which, in turn, is verified if $Q \preceq -2L$ and $2D \succeq L$ (see Fact \ref{Fact:psd}). Note that, by invoking Lemma 1, $Q \preceq -2L$ is implied by \textcolor{magenta}{$Q\preceq -4D$}, and $2D \succeq L$ is always satisfied.	
	By re-writing \eqref{firstLMI} as
	\begin{equation}
		\begin{bmatrix}
			I \\ L
		\end{bmatrix}\' \begin{bmatrix}
			-\tfrac{1}{2}Q & S \\ \star & \tfrac{1}{2}D^{-1}-R
		\end{bmatrix} \begin{bmatrix}
			I \\ L
		\end{bmatrix} \succ 0,
	\end{equation}
	we can conclude that it is verified if (see Fact \ref{Fact:congruence})
	\begin{equation}
		\begin{bmatrix}
			-\tfrac{1}{2}Q & S \\ \star & \tfrac{1}{2}D^{-1}-R
		\end{bmatrix} \succ 0,
	\end{equation}
	which is implied by \textcolor{magenta}{$S\succeq 0$}, \textcolor{magenta}{$Q \prec -2S$} and \textcolor{magenta}{$R+S\prec \tfrac{1}{2}D\inv$}  (see Fact \ref{Fact:psd}).  \hfill \QED

\end{appendices}

%%%%%%%%%%%%%%%%%%%%%%%%%%%%%%%%%%%%%%%%%%%%%%%%%%%%%%%%%%%%

\begin{IEEEbiography}[{\includegraphics[trim={5cm 0 5cm 0},width=1in,height=1.25in,clip,keepaspectratio]{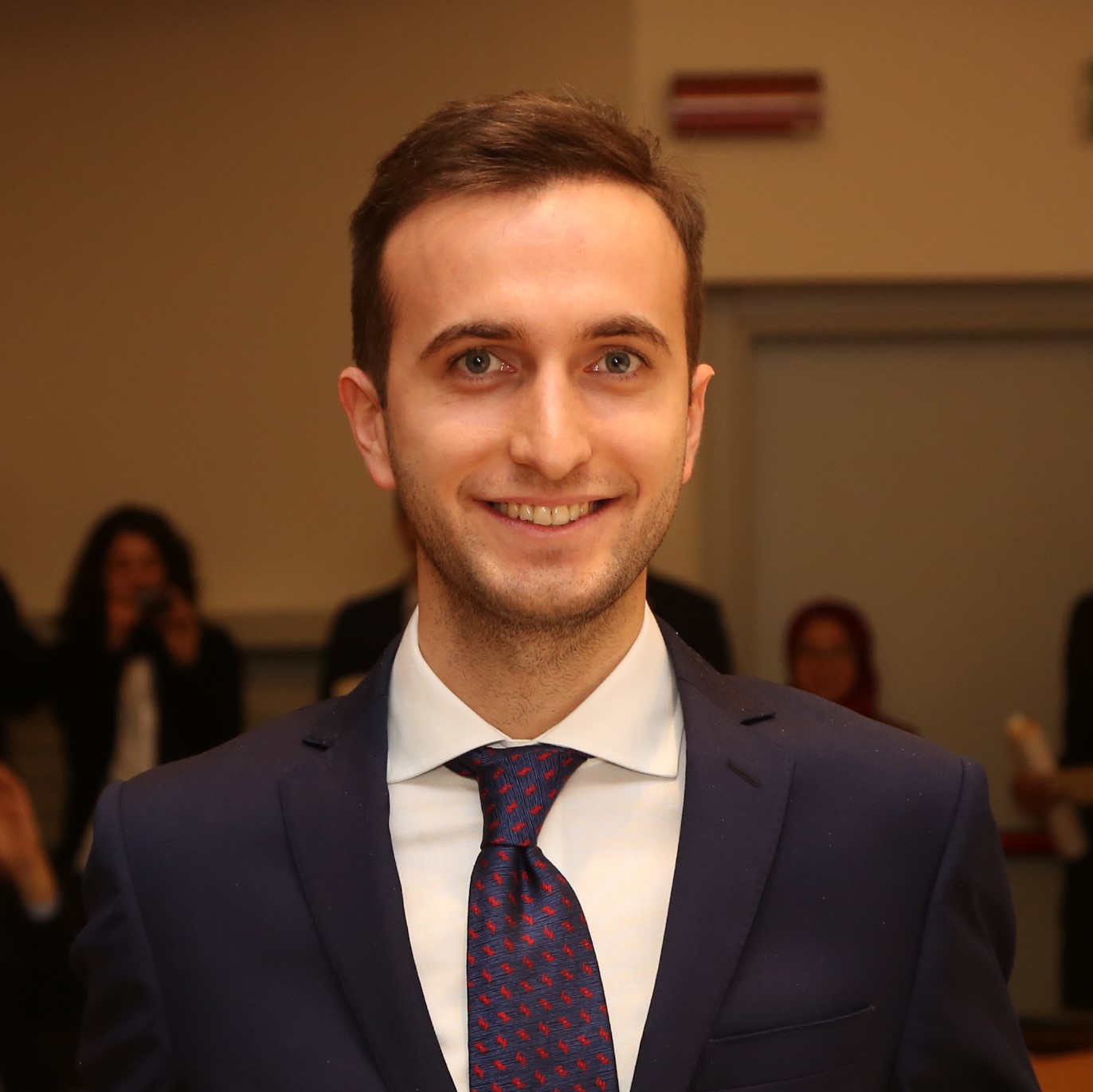}}]{Andrea Martinelli}
	received the B.Sc. degree in management engineering and the M.Sc. degree in control engineering from Politecnico di Milano, Italy, in 2015 and 2017, respectively. He obtained a PhD degree at the Automatic Control Laboratory, ETH Zürich, Switzerland, in 2024, where he currently holds a postdoctoral position. In 2017, he conducted his M.Sc. thesis at the Laboratoire d'Automatique, EPF Lausanne, Switzerland. In 2018, he was a Research Assistant with the Systems \& Control Group at Politecnico di Milano. His research interests focus on data-driven optimal control and control of interconnected systems. 
\end{IEEEbiography}

\begin{IEEEbiography}[{\includegraphics[width=1in,height=1.25in,clip,keepaspectratio]{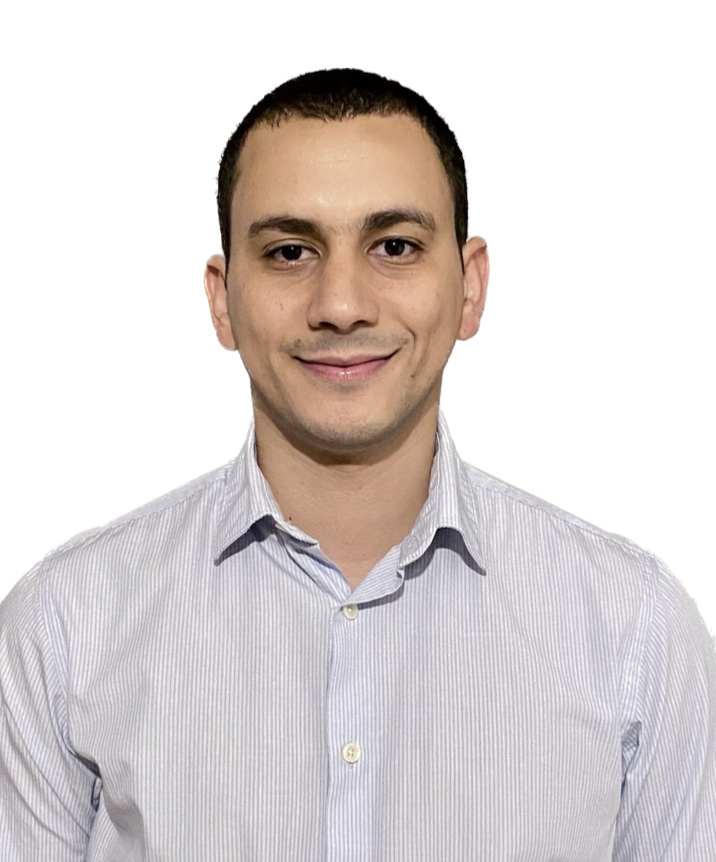}}]{Ahmed Aboudonia} is a postdoctoral researcher with the Automatic Control Laboratory (IfA) at ETH Zurich. He received the PhD degree in electrical engineering from ETH Zurich in 2023, the M.Sc. degree in control engineering from Sapienza University of Rome in 2018 and the B.Sc. degree in mechatronics engineering from the German University in Cairo in 2014. His research interests lie in the intersection of control theory, mathematical optimization and machine learning with application to robotics, energy and transportation systems.
	
\end{IEEEbiography}

\begin{IEEEbiography}[{\includegraphics[width=1in,height=1.25in,clip,keepaspectratio]{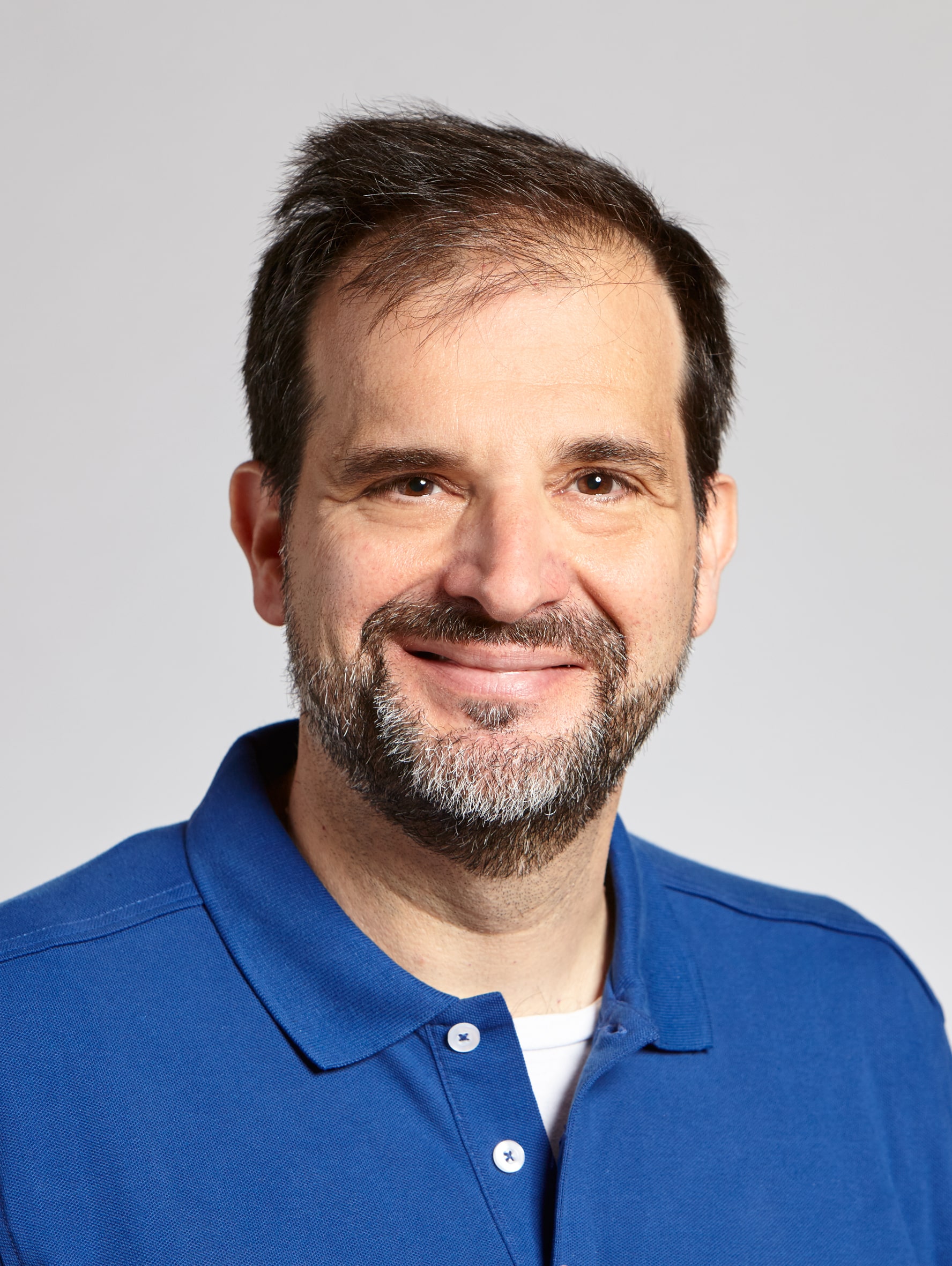}}]{John Lygeros} received a B.Eng. degree in 1990 and an M.Sc. degree in 1991 from Imperial College, London, U.K. and a Ph.D. degree in 1996 at the University of California, Berkeley. After research appointments at M.I.T., U.C. Berkeley and SRI International, he joined the University of Cambridge in 2000 as a University Lecturer. Between March 2003 and July 2006 he was an Assistant Professor at the Department of Electrical and Computer Engineering, University of Patras, Greece. In July 2006 he joined the Automatic Control Laboratory at ETH Zurich where he is currently serving as the Professor for Computation and Control and the Head of the laboratory. His research interests include modelling, analysis, and control of large-scale systems, with applications to biochemical networks, energy systems, transportation, and industrial processes. John Lygeros is a Fellow of IEEE, and a member of IET and the Technical Chamber of Greece. Between 2013 and 2023 he served as the Vice-President Finances and a Council Member of the International Federation of Automatic Control. Since 2020 he is serving as the Director of the National Center of Competence in Research ``Dependable Ubiquitous Automation" (NCCR Automation).
\end{IEEEbiography}

\end{document}